\documentclass[12pt]{amsart}
\usepackage{amsmath}
\usepackage{amsfonts}
\usepackage{amssymb}
\usepackage{bbding}
\usepackage{stmaryrd}
\usepackage{txfonts}
\usepackage{graphicx}
\usepackage{epsfig}
\usepackage{xypic}
\usepackage{tikz}
\usepackage{longtable}
\usepackage[all]{xy}
\usepackage{pgflibraryarrows}
\usepackage{pgflibrarysnakes}
\usepackage[shortlabels]{enumitem}
\usepackage{ifpdf}
\ifpdf
  \usepackage[colorlinks,final,backref=page,hyperindex]{hyperref}
\else
  \usepackage[colorlinks,final,backref=page,hyperindex,hypertex]{hyperref}
\fi
\usepackage{xspace}

\newtheorem{theorem}{Theorem}[section]
\newtheorem{lemma}[theorem]{Lemma}
\newtheorem{coro}[theorem]{Corollary}

\newtheorem{prop}[theorem]{Proposition}
\theoremstyle{definition}
\newtheorem{defn}[theorem]{Definition}
\newtheorem{remark}[theorem]{Remark}
\newtheorem{exam}[theorem]{Example}

\newcommand{\nc}{\newcommand}
\newcommand{\delete}[1]{}

\nc{\tred}[1]{\textcolor{red}{#1}}
\nc{\tblue}[1]{\textcolor{blue}{#1}} \nc{\tgreen}[1]{\textcolor{green}{#1}} \nc{\tpurple}[1]{\textcolor{purple}{#1}} \nc{\btred}[1]{\textcolor{red}{\bf #1}} \nc{\btblue}[1]{\textcolor{blue}{\bf #1}} \nc{\btgreen}[1]{\textcolor{green}{\bf #1}} \nc{\btpurple}[1]{\textcolor{purple}{\bf #1}}

\newcommand{\efootnote}[1]{}

\nc{\mlabel}[1]{\label{#1}}  
\nc{\mcite}[2][]{\cite[#1]{#2}}  
\nc{\mref}[1]{\ref{#1}}  
\nc{\mbibitem}[1]{\bibitem{#1}} 

\delete{
\nc{\mlabel}[1]{\label{#1}  
{\hfill \hspace{1cm}{\bf{{\ }\hfill(#1)}}}}
\nc{\mcite}[1]{\cite{#1}}  
\nc{\mref}[1]{\ref{#1}{{\bf{{\ }(#1)}}}}  
\nc{\mbibitem}[1]{\bibitem[\bf #1]{#1}} 
}

\renewcommand\geq{\geqslant}
\renewcommand\leq{\leqslant}

\renewcommand\bar[1]{\overline{#1}}

\nc{\tforall}{\quad \text{ for all }}

\nc{\nz}{\varepsilon}
\nc{\Id}{\mathrm{Id}}
\nc{\map}[2]{{#2}^{#1}}
\nc{\gp}{B}

\nc{\name}[1]{{\bf #1}}

\nc{\Irr}{\mathrm{Irr}}
\nc{\vx}{\sigma} \nc{\vy}{\tau} \nc{\dvx}{\sigma^{(1)}} \nc{\dvy}{\tau^{(1)}} \nc{\done}{\vep} \nc{\mcitep}[1]{\mcite{#1}} \nc{\wt}{\mathrm{wt}} \nc{\bre}[1]{|#1|} \nc{\mapmonoid}{\frakM} \nc{\disjoint}{\frakM'}
\nc{\ncpoly}[1]{\langle #1\rangle}  
\nc{\mapm}[1]{\lfloor\!|{#1}|\!\rfloor}
\nc{\diff}[1]{{}^\NC\{ #1 \}} \nc{\disj}[1]{\{{#1}\}'} \nc{\mdisj}[1]{\frakM'(#1)} \nc{\brho}{\bar{\rho}} \nc{\om}{\bar{\frakm}} \nc{\frakn}{\mathfrak n} \nc{\ddeg}[1]{^{(#1)}} \nc{\opset}{X} \nc{\genset}{{Z}} \nc{\NC}{\mathrm{{NC}}} \nc{\leaf}{\mathrm{leaf}} \nc{\twig}{\mathrm{twig}} \nc{\fe}{\mathrm{fl}} \nc{\munderline}[1]{#1} \nc{\bo}{o} \nc{\dep}{\mathrm{depth}} \nc{\ofe}{\mathrm{ofl}} \nc{\dfe}{\mathrm{dfe}} \nc{\fex}{\mathrm{fex}} \nc{\dl}{\mathrm{dlex}} \nc{\db}{\mathrm{db}} \nc{\lex}{\mathrm{lex}} \nc{\clex}{\mathrm{clex}} \nc{\dgp}{\mathrm{dgp}} \nc{\dgx}{\mathrm{dgx}} \nc{\br}{\mathrm{br}} \nc{\obd}{\mathrm{odb}} \nc{\ob}{\mathrm{ob}}
\nc{\pie}{\mathrm{PIE}}
\nc{\rbo}{\mathrm{RBO}}
\nc{\supp}{\mathcal{S}}
\nc{\nul}{\mathcal{Z}}

\nc{\bin}[2]{ (_{\stackrel{\scs{#1}}{\scs{#2}}})}  
\nc{\binc}[2]{ \left (\!\! \begin{array}{c} \scs{#1}\\
    \scs{#2} \end{array}\!\! \right )}  
\nc{\bincc}[2]{  \left ( {\scs{#1} \atop
    \vspace{-1cm}\scs{#2}} \right )}  
\nc{\bs}{\bar{S}} \nc{\cosum}{\sqsubset} \nc{\la}{\longrightarrow} \nc{\rar}{\rightarrow} \nc{\dar}{\downarrow} \nc{\dprod}{**} \nc{\dap}[1]{\downarrow \rlap{$\scriptstyle{#1}$}} \nc{\md}[1]{\bar{#1}} \nc{\uap}[1]{\uparrow \rlap{$\scriptstyle{#1}$}} \nc{\defeq}{\stackrel{\rm def}{=}} \nc{\disp}[1]{\displaystyle{#1}} \nc{\dotcup}{\ \displaystyle{\bigcup^\bullet}\ } \nc{\gzeta}{\bar{\zeta}} \nc{\hcm}{\ \hat{,}\ } \nc{\hts}{\hat{\otimes}} \nc{\barot}{{\otimes}} \nc{\free}[1]{\bar{#1}} \nc{\uni}[1]{\tilde{#1}} \nc{\hcirc}{\hat{\circ}} \nc{\leng}{\ell} \nc{\lleft}{[} \nc{\lright}{]} \nc{\lc}{\lfloor} \nc{\rc}{\rfloor}
\nc{\lb}{[} 
\nc{\rb}{]} 
\nc{\curlyl}{\left \{ \begin{array}{c} {} \\ {} \end{array}
    \right.  \!\!\!\!\!\!\!}
\nc{\curlyr}{ \!\!\!\!\!\!\!
    \left. \begin{array}{c} {} \\ {} \end{array}
    \right \} }
\nc{\longmid}{\left | \begin{array}{c} {} \\ {} \end{array}
    \right. \!\!\!\!\!\!\!}
\nc{\onetree}{\bullet} \nc{\ora}[1]{\stackrel{#1}{\rar}}
\nc{\ola}[1]{\stackrel{#1}{\la}}
\nc{\ot}{\otimes} \nc{\mot}{{{\boxtimes\,}}} \nc{\otm}{\overline{\boxtimes}} \nc{\sprod}{\bullet} \nc{\scs}[1]{\scriptstyle{#1}} \nc{\mrm}[1]{{\rm #1}} \nc{\msum}{\sum\limits}
\nc{\margin}[1]{\marginpar{\rm #1}}   
\nc{\dirlim}{\displaystyle{\lim_{\longrightarrow}}\,} \nc{\invlim}{\displaystyle{\lim_{\longleftarrow}}\,} \nc{\mvp}{\vspace{0.3cm}} \nc{\tk}{^{(k)}} \nc{\tp}{^\prime} \nc{\ttp}{^{\prime\prime}} \nc{\svp}{\vspace{2cm}} \nc{\vp}{\vspace{8cm}} \nc{\proofbegin}{\noindent{\bf Proof: }}
\nc{\proofend}{$\blacksquare$ \vspace{0.3cm}}
\nc{\modg}[1]{\!<\!\!{#1}\!\!>}
\nc{\intg}[1]{F_C(#1)} \nc{\lmodg}{\!<\!\!} \nc{\rmodg}{\!\!>\!} \nc{\cpi}{\widehat{\Pi}}
\nc{\sha}{{\mbox{\cyr X}}}  
\nc{\shap}{{\mbox{\cyrs X}}} 
\nc{\shpr}{\diamond}    
\nc{\shp}{\ast} \nc{\shplus}{\shpr^+}
\nc{\shprc}{\shpr_c}    
\nc{\msh}{\ast} \nc{\zprod}{m_0} \nc{\oprod}{m_1} \nc{\vep}{\varepsilon} \nc{\labs}{\mid\!} \nc{\rabs}{\!\mid}
\nc{\astarrow}{\overset{\raisebox{-3pt}{$\ast$}}{\rightarrow}}

\nc{\Sym}{\mrm{Sym}}
\nc{\NSym}{\mrm{NSym}}
\nc{\QSym}{\mrm{QSym}}
\nc{\RQSym}{\mrm{RQSym}}
\nc{\SSym}{\mrm{\mathfrak{S}Sym}}
\nc{\HSym}{\mrm{\mathfrak{H}Sym}}
\nc{\syms}{symmetric functions\xspace}
\nc{\qsyms}{quasi-symmetric functions\xspace}
\nc{\nsymg}{\mathrm{NSym}_\gp}
\nc{\parr}{{\rm {Par}}}
\nc{\Des}{{\rm {Des}}}
\nc{\set}{{\rm {set}}}
\nc{\comp}{{\rm {comp}}}
\nc{\Ap}{{\rm {Ap}}}
\nc{\wcomp}{{\rm {rcomp}}}
\nc{\dom}{{\rm{dom}}}
\nc{\ran}{{\rm{ran}}}
\nc{\rk}{{\rm{rk}}}
\nc{\ass}{{\rm{ass}}}
\nc{\st}{{\rm{st}}}
\nc{\wcts}{|\hspace{-1.2mm}\models}

\nc{\dth}{d} \nc{\mmbox}[1]{\mbox{\ #1\ }} \nc{\fp}{\mrm{FP}} \nc{\rchar}{\mrm{char}} \nc{\Fil}{\mrm{Fil}} \nc{\Mor}{Mor\xspace} \nc{\gmzvs}{gMZV\xspace} \nc{\gmzv}{gMZV\xspace} \nc{\mzv}{MZV\xspace} \nc{\mzvs}{MZVs\xspace} \nc{\Hom}{\mrm{Hom}} \nc{\id}{\mrm{id}} \nc{\im}{\mrm{im}} \nc{\incl}{\mrm{incl}}  \nc{\mchar}{\rm char}

\nc{\Alg}{\mathbf{Alg}} \nc{\Bax}{\mathbf{Bax}} \nc{\bff}{\mathbf f} \nc{\bfk}{{\bf k}} \nc{\bfone}{{\bf 1}} \nc{\bfx}{\mathbf x} \nc{\bfy}{\mathbf y}
\nc{\base}[1]{\bfone^{\otimes ({#1}+1)}} 
\nc{\Cat}{\mathbf{Cat}} \delete{}
\nc{\detail}{\marginpar{\bf More detail}
    \noindent{\bf Need more detail!}
    \svp}
\nc{\Int}{\mathbf{Int}} \nc{\Mon}{\mathbf{Mon}}
\nc{\rbtm}{{shuffle }} \nc{\rbto}{{Rota-Baxter }} \nc{\remarks}{\noindent{\bf Remarks: }} \nc{\Rings}{\mathbf{Rings}} \nc{\Sets}{\mathbf{Sets}}
\nc{\balpha}{\mathbf{\alpha}}

\nc{\BA}{{\mathbb A}} \nc{\CC}{{\mathbb C}} \nc{\DD}{{\mathbb D}} \nc{\EE}{{\mathbb E}} \nc{\FF}{{\mathbb F}} \nc{\GG}{{\mathbb G}} \nc{\HH}{{\mathbb H}} \nc{\LL}{{\mathbb L}} \nc{\NN}{{\mathbb N}} \nc{\KK}{{\mathbb K}} \nc{\PP}{{\mathbb P}} \nc{\QQ}{{\mathbb Q}} \nc{\RR}{{\mathbb R}} \nc{\TT}{{\mathbb T}} \nc{\VV}{{\mathbb V}} \nc{\ZZ}{{\mathbb Z}}

\nc{\cala}{{\mathcal A}} \nc{\calc}{{\mathcal C}} \nc{\cald}{{\mathcal D}} \nc{\cale}{{\mathcal E}} \nc{\calf}{{\mathcal F}} \nc{\calg}{{\mathcal G}} \nc{\calh}{{\mathcal H}} \nc{\cali}{{\mathcal I}} \nc{\call}{{\mathcal L}} \nc{\calm}{{\mathcal M}} \nc{\caln}{{\mathcal N}} \nc{\calo}{{\mathcal O}} \nc{\calp}{{\mathcal P}} \nc{\calr}{{\mathcal R}} \nc{\cals}{{\mathcal S}} \nc{\calt}{{\mathcal T}} \nc{\calw}{{\mathcal W}} \nc{\calk}{{\mathcal K}} \nc{\calx}{{\mathcal X}}
\nc{\calz}{{\mathcal Z}}
 \nc{\CA}{\mathcal{A}}

\nc{\fraka}{{\mathfrak a}} \nc{\frakA}{{\mathfrak A}} \nc{\frakb}{{\mathfrak b}} \nc{\frakB}{{\mathfrak B}}
\nc{\frakc}{{\mathfrak c}}  \nc{\frakD}{{\mathfrak D}}
\nc{\frakH}{{\mathfrak H}}
\nc{\frakh}{{\mathfrak h}} \nc{\frakM}{{\mathfrak M}}
\nc{\frakO}{{\mathfrak O}}
\nc{\frakE}{{\mathfrak E}}
\nc{\bfrakM}{\overline{\frakM}} \nc{\frakm}{{\mathfrak m}} \nc{\frakP}{{\mathfrak P}} \nc{\frakN}{{\mathfrak N}} \nc{\frakp}{{\mathfrak p}} \nc{\frakS}{{\mathfrak S}}
\nc{\frakk}{{\mathfrak k}}
\nc{\frakx}{{\mathfrak x}}
\nc{\frakl}{{\mathfrak l}} \nc{\ox}{\bar{\frakx}} \nc{\frakX}{{\mathfrak X}} \nc{\fraky}{{\mathfrak y}} \nc\dop{\delta}
\nc{\Reduce}{{\rm Red}}

\font\cyr=wncyr10 \font\cyrs=wncyr7
\nc{\redt}[1]{\textcolor{red}{#1}}
\nc{\li}[1]{\textcolor{red}{#1}}
\nc{\lir}[1]{\textcolor{red}{\tt Li: #1}}
\nc{\yur}[1]{\textcolor{blue}{\tt Yu:#1}}

\begin{document}
\title[The Hopf algebras of signed permutations and weak compositions]{The Hopf algebras of signed permutations, of weak quasi-symmetric functions and of Malvenuto-Reutenauer}

\author{Li Guo}
\address{
Department of Mathematics and Computer Science, Rutgers University, Newark, NJ 07102, USA}
\email{liguo@rutgers.edu}

\author{Jean-Yves Thibon}
\address{Laboratoire d'Informatique Gaspard Monge, Universit\'e Paris-Est Marne-la-Vall\'ee, 5 Boulevard Descartes,
Champs-sur-Marne, 77454 Marne-la-Vall\'ee cedex 2, France}
\email{jyt@univ-mlv.fr}

\author{Houyi Yu}
\address{School of Mathematics and Statistics, Southwest University, Chongqing 400715, China}
\email{yuhouyi@swu.edu.cn}

\hyphenpenalty=8000
\date{\today}

\begin{abstract}
This paper builds on two covering Hopf algebras of the Hopf algebra $\QSym$ of quasi-symmetric functions, with linear bases parameterized by compositions. One is the Malvenuto-Reutenauer Hopf algebra $\SSym$ of permutations, mapped onto $\QSym$ by taking descents of permutations. The other one is the recently introduced Hopf algebra $\RQSym$ of weak quasi-symmetric functions, mapped onto $\QSym$ by extracting compositions from weak compositions.

We extend these two surjective Hopf algebra homomorphisms into a commutative diagram by introducing a Hopf algebra $\HSym$, linearly spanned by signed permutations from the hyperoctahedral groups, equipped with the shifted quasi-shuffle product and deconcatenation coproduct. Extracting a permutation from a signed permutation defines a Hopf algebra surjection form $\HSym$ to $\SSym$ and taking a suitable descent from a signed permutation defines a linear surjection from $\HSym$ to $\RQSym$.
The notion of signed $P$-partitions from signed permutations is introduced which, by taking generating functions, gives fundamental weak quasi-symmetric functions and sends the shifted quasi-shuffle product to the product of the corresponding generating functions.
Together with the existing Hopf algebra surjections from $\SSym$ and $\RQSym$ to $\QSym$, we obtain a commutative diagram of Hopf algebras revealing the close relationship among compositions, weak compositions, permutations and signed permutations.
\end{abstract}

\subjclass[2010]{
05E05,  
16T30,  
16W99,  
}

\keywords{quasi-symmetric function, Malvenuto-Reutenauer Hopf algebra, weak quasi-symmetric function, signed permutation, signed $P$-partition, quasi-shuffle product, Hopf algebra}

\maketitle

\vspace{-.8cm}

\tableofcontents

\vspace{-.8cm}


\allowdisplaybreaks

\section{Introduction}\label{sec:int}
This paper introduces the Hopf algebra of signed permutations and establishes its relationship with the Hopf algebras of permutations and of weak quasi-symmetric functions.

\subsection{Hopf algebras of various symmetric functions}

Beginning with the pioneering work of Joni and Rota
~\cite{JR79}, many combinatorial objects have been endowed with Hopf algebra structures which, while giving precise algebraic meanings of combinatorial operations, provided intuitive realizations of abstract concepts. Furthermore, relationship between different combinatorial objects are made precise as homomorphisms between these Hopf algebras through which properties in one Hopf algebra become better understood through the interaction with other Hopf algebras.

Central among these ``combinatorial Hopf algebras" are the Hopf algebras $\Sym$ of symmetric functions~\cite{Gl77},  $\QSym$ of quasi-symmetric functions~\cite{Ge84},
$\NSym$ of  noncommutative symmetric functions~\cite{GKLRT95} and $\SSym$ of permutations~\cite{Mal94,MRe95}.
Homomorphisms among these Hopf algebras are summarized in the commutative diagram

\begin{equation}
\begin{split}
\xymatrix{
\SSym \ar@{->>}[d]_{\mathcal{D}_1} &&\ \NSym \ar@{>->}[ll]  \ar@{->>}[d] \\
\QSym 
&& \ \text{Sym} \ar@{>->}[ll]
}
\mlabel{eq:diag1}
\end{split}
\end{equation}
demonstrating the close relations among partitions, compositions and permutations.
For instance, the multiplication of two fundamental quasi-symmetric functions in $\QSym$ is made much simpler when the functions are lifted to $\SSym$ and multiplied there by means of the shifted shuffle product.

Quasi-symmetric functions were formally introduced by Gessel~\cite{Ge84} in 1984 as a source of generating functions for Stanley's ordinary $P$-partitions~\cite{St1972}, even though the related ideas can be tracked back much earlier.
Since then they have found applications in diverse areas in mathematics including Hopf algebras~\cite{Ehr96}, representation theory~\cite{Hi00}, number theory~\cite{Ho00,HI17}, algebraic topology~\cite{BR08}, discrete geometry~\cite{BHW03} and probability~\cite{HH09}. The pivotal role played by quasi-symmetric functions was conceptualized by the property
that the Hopf algebra of quasi-symmetric functions is the terminal object in the category of combinatorial Hopf algebras~\cite{ABS06}.

The space of quasi-symmetric functions is canonically spanned by monomial quasi-symmetric functions parameterized by compositions, that is, by vectors of positive integers.
In order to study generalizations of symmetric functions in the context of Rota-Baxter algebras envisioned by Rota~\cite{Ro95}, the authors introduced in the recent paper~\cite{YGT} the Hopf algebra $\RQSym$ of
{\bf weak quasi-symmetric functions}, linearly spanned by monomial weak quasi-symmetric functions parameterized by weak compositions, that is, by vectors of non-negative integers. To deal with the divergence of such a function caused by the zero entries in a weak composition, the notion of  a {\bf regularized composition} (called an $\widetilde{\NN}$-composition in~\cite{YGT}) was introduced to give a ``regularization" of the weak quasi-symmetric functions. It is shown that $\QSym$ is at the same time a Hopf subalgebra and a Hopf quotient algebra of $\RQSym$ via a Hopf algebra surjection $\varphi_1:\RQSym \to \QSym$, leading to the following extension of the commutative diagram in Eq.~\eqref{eq:diag1}:
\begin{equation}
\begin{split}
\xymatrix{
\HSym \atop \text{(signed permutations)} \ar@{-->>}[rr]^{\varphi_2}\ar@{-->>}[d]_{\mathcal{D}_2}^{\text{weak descents \&} \atop \text{signed } P-\text{partitions}}
&&\SSym \atop \text{(permutations)} \ar@{->>}[d]_{\mathcal{D}_1}^{\text{descents \&}\atop P\text{-partitions}} &&\ \NSym \ar@{>->}[ll]  \ar@{->>}[d] \\
  \RQSym \atop  \text{(weak compositions)} \ar@{->>}[rr]^{\varphi_1}
&&\QSym \atop  \text{(compositions)} && \ \text{Sym} \ar@{>->}[ll]
}
\mlabel{eq:diag2}
\end{split}
\end{equation}
which suggests a ``pullback" Hopf algebra $\HSym$ together with suitable Hopf algebra homomorphisms $\varphi_2$ and $\mathcal{D}_2$.

In a nutshell, the purpose of this paper is to complete this square.
We next give an outline of the paper centered around this diagram.

\subsection{Outline of the paper}

The Malvenuto-Reutenauer Hopf algebra~\cite{Mal94, MRe95}, also known as the Hopf algebra of free quasi-symmetric functions~\cite{DHT02}, is  $\SSym=\bigoplus_{n\geq0} \bfk \mathfrak{S}_n$, linearly spanned by permutations in the symmetric groups $\frakS_n$ and equipped with the shifted shuffle product of permutations.

Signed permutations on $n$ letters naturally form a group under composition, called the $n$-th hyperoctahedral group and denoted by $\mathfrak{B}_n$. This is recalled in Section~~\ref{sec:hopfalgonW} in order to obtain a Hopf algebra  $\HSym=\bigoplus_{n=0}^{\infty} \bfk \mathfrak{B}_n$ of signed permutations (Theorem~\ref{thm:WSisahopfalg}).
Here the multiplication turns out to be defined by a shifted quasi-shuffle product of signed permutation, which calls for a quasi-shuffle algebra with a weight and with the generating algebra not necessarily commutative.

In Section~\ref{sec:BaP} we first give some background on quasi-symmetric functions and weak quasi-symmetric functions.
The Hopf algebra homomorphism $\cald_1: \SSym \to \QSym$ in~\cite{Mal94,MRe95} was defined by taking descents of permutations. By taking weak descents, each signed permutation is associated with a unique regularized composition, giving rise to a linear map $\cald_2:\HSym \to \RQSym$.
The Hopf algebra surjection $\varphi_1: \RQSym \to \QSym$ in~\cite{YGT} was defined by extracting a composition from a regularized composition. Likewise, we define a linear surjection $\varphi_2:\HSym \to \SSym$ by extracting a permutation from a signed permutation.
Then the commutativity of the resulting left square in diagram~(\ref{eq:diag2}) is verified in Theorem~\ref{thm:comm}, on the level of linear maps.

The fact that $\varphi_2$ is a Hopf algebra homomorphism is proved in Theorem \ref{thm:phi2}, in which the interpretation of the shifted quasi-shuffle product as a stuffle product plays a key role.

In order to verify that $\cald_1$ is a Hopf algebra homomorphism~\cite{Mal94,MRe95}, the notion of $P$-partitions was used to realize the fundamental quasi-symmetric functions as generating functions for partitions of permutations, showing that the shifted shuffle product of two permutations is sent to the product of the two corresponding fundamental quasi-symmetric functions.

To apply this idea to verify the Hopf algebra homomorphism property of $\cald_2$, we need a suitable generalization of $P$-partitions.
There are a number of generalizations of $P$-partitions. Stembridge~\cite{Ste97} outlined how the generating functions
of enriched $P$-partitions give peak quasi-symmetric functions.
Chow~\cite{Cho01} studied type $B$ quasi-symmetric functions using ordinary type $B$ $P$-partitions which is a simpler version of the notion due to Reiner~\cite{Re93}.
Petersen~\cite{PK07} connected type $B$  enriched $P$-partitions to type $B$ peak algebras. Lam and Pylyavskyy \cite{LP07} defined set-valued
P-partitions, which are maps that take nonempty finite  sets of positive integers as values.

For our purpose, we introduce another generalization of $P$-partitions. They are called signed $P$-partitions and are defined from signed permutations.
We show that the generating function for the signed $P$-partition of a signed permutation is precisely the fundamental weak quasi-symmetric function indexed by the regularized composition coming from the signed permutation by taking weak descents (Theorem \ref{thm:Gamma(pi)=Fwcomp(pi)}). Thus this process coincides with the linear map $\cald_2$. Further this process sends the shifted quasi-shuffle product of signed permutations to the product of the corresponding fundamental weak quasi-symmetric functions, thereby showing that $\cald_2$ is a Hopf algebra surjection  (Theorem~\ref{HAHD2fHSymtRQSym}).

\smallskip

\noindent
{\bf Notations.} Unless otherwise specified, an algebra in this paper is assumed to be defined over a field $\bfk$ of characteristic zero.
Denote by $\mathbb{N}$ the set of nonnegative integers and $\mathbb{P}$ the set of positive integers.
Given any $m, n \in \PP$ with $m\leq n$, let $[m,n]=\{m, m +1,\cdots, n\}$
and $[n]=[1, n]$.
For a set $A$ of positive integers, the poset notation $A=\{a_1<a_2<\cdots<a_k\}$ indicates that $A=\{a_1,a_2,\cdots,a_k\}$ and $a_1<a_2<\cdots<a_k$.

\section{Hopf algebra of signed permutations}
\label{sec:signperm}
\label{sec:hopfalgonW}
\label{subsec:Partialpermutations}

In this section we show that the linear space of signed permutations carries a Hopf algebra structure that naturally extends the Malvenuto-Reutenauer Hopf algebra of permutations.

The Malvenuto-Reutenauer Hopf algebra~\cite{Mal94, MRe95} $\SSym=\bigoplus_{n\geq0} \bfk \mathfrak{S}_n$ is linearly spanned by the permutations $\mathfrak{S}_n$ on $[n], n\geq 0$, with the convention that $\mathfrak{S}_0=\{\imath\}$, $\imath$ being the identity.
The elements $\pi\in\mathfrak{S}_n$ are viewed as words $\pi=(\pi(1),\pi(2),\cdots, \pi(n))$ or simply $\pi_1\pi_2\cdots\pi_n$.
If $m$ is a positive integer, then let $\pi[m]:=(\pi_1+m)(\pi_2+m)\cdots(\pi_n+m)$,
regarded as a permutation of $[m+1,m+n]$.

For two permutations $\sigma\in \mathfrak{S}_m$ and $\tau\in \mathfrak{S}_n$, their \name{shifted shuffle}, denoted by $\sigma\overline{\shap}\tau$, is the shuffle of $\sigma$
and $\tau[m]$,
so $\sigma\overline{\shap}\tau=\sigma \shap \tau[m]$ is the sum of permutations in $\mathfrak{S}_{m+n}$
whose first $m$ smaller values are in the same relative order as $\sigma$ and whose last $n$ larger values are in the same relative order as $\tau[m]$.
For example, we have
$$12\overline{\shap}12=12\shap34= 1234+1324+1342+3124+3142+3412.$$
Then $\SSym$ is a Hopf algebra with the shifted shuffle product $\overline{\shap}$ and the shifted deconcatenation coproduct $\Delta:=(\st\otimes\st)\delta$, where $\delta$ is the deconcatenation coproduct and $\st$ is the usual standardization on words, replacing the $i$-th smallest element of a word by $i$. More precisely,
\begin{align}
   \Delta(\pi_1\cdots \pi_n)=\sum_{p=0}^m \st(\pi_1\cdots\pi_p)\otimes \st(\pi_{p+1}\cdots\pi_n).
\mlabel{eq:delta}
\end{align}
For instance,
$\Delta(1324)=\imath\otimes 1324+1\otimes 213+12\otimes 12
   +132\otimes 1+1324\otimes \imath$.
See \cite{AS05} for further background on the Malvenuto-Reutenauer Hopf algebra.

Let $[n]^\pm$ denote the set $\{n^{-},\cdots,2^{-},1^{-},0,1,$ $2,\cdots,n\}$, where the additive inverse $-n$ of an integer $n$ is denoted by $n^{-}$  for notational clarity.
A \name{signed permutation} of $[n]$ is a permutation  $\pi$ of the set $[n]^\pm$ such that $\pi({i}^{-})={\pi(i)}^{-}$ for $i\in[0,n]$. This notion has many applications in algebra and combinatorics~\cite{BB05,MR95,Re93s}.
Signed permutations of $[n]$ naturally form a group under
composition, called the $n$-th \name{hyperoctahedral group} and denoted by $\mathfrak{B}_n$, which is a Coxeter group of type $B$.
We identify the $n$-th symmetric group $\mathfrak{S}_n$ with the subgroup of $\mathfrak{B}_n$ consisting of all signed permutations $\pi$ such that $\pi([n])=[n]$.
We will equip $\HSym=\bigoplus_{n\geq0}\bfk\mathfrak{B}_n$
with a product and a coproduct which give a Hopf algebra structure on $\HSym$, such that $\SSym$ is a Hopf subalgebra and a Hopf quotient algebra of $\HSym$.

For a signed permutation $\pi$, it is easy to see that $\pi(0)=0$.
So $\pi$ is uniquely determined by the images of $1,2,\cdots,n$. So we will write $\pi=(\pi(1),\pi(2),\cdots,\pi(n))$ or simply  $\pi=\pi_1\pi_2\cdots\pi_n$. The number $n$ is called its \name{length} and denoted by $\ell(\pi)$. For example, $\pi=53^{-}246^{-}1^{-}$ is the signed permutation
that maps $1$ to $5$, $2$ to ${3}^{-}$, $3$ to $2$, and so on.

Motivated by the notions of standard permutation \cite{Reu93} and  colored standardization \cite{NT10},
for a word $w=a_1a_2\cdots a_n$ on the alphabet set $\mathbb{Z}\backslash\{0\}$,  we define its \name{standard signed permutation} to be the unique signed permutation
${\st}(w)=b_1b_2\cdots b_n\in \mathfrak{B}_n$ such that
\begin{enumerate}
\item for $i\in [n]$, $b_i$ and $a_i$ have the same sign,
\item for $i, j\in[n]$, if $|a_i|<|a_j|$ or $|a_i|=|a_j|$ with $i<j$, then $|b_i|<|b_j|$,
\end{enumerate}
Here $|a_i|$ denotes the absolute value of $a_i$.
For instance, $\st(3{2}^{-}7{5}^{-})=2{1}^{-}4{3}^{-}$ and $\st(2{2}^{-}1{2}^{-}2)=2{3}^{-}1{4}^{-}5$.

To extend the shifted shuffle product $\overline{\shap}$ on the space $\SSym$ of permutations to the space $\HSym$ of signed permutations, we make use of the quasi-shuffle product, by including a weight and removing the commutativity condition on the generating algebra. See also~\cite{CGPZ}.

Let $A$ be a set whose elements are called \name{alphabets} or \name{letters}.
Let $A^*$ denote the free monoid generated by $A$ whose elements are called {\bf words} on the letter set $A$, including the the empty sequence, denoted by $\imath$ and called the \name{empty word}.

Given a word $w=a_1a_2\cdots a_n$ on $A$ with $a_i\in A$ for $i\in [n]$, we call $a_i$ the parts of $w$ and
$n$ the length of $w$, denoted by $\ell(w)$. The length of the empty word $\imath$ is defined to be $0$.
Then the linear space $\bfk A^*$ with the concatenation product is the free associative algebra $\bfk\langle A\rangle$ on $A$.
Next suppose that there is an associative product $\bullet$ on $\bfk A$.
Fix $\lambda\in \bfk$.
We define a new $\bfk$-bilinear multiplication $\star_\lambda=\star_{\lambda,\bullet}$ on $\bfk\langle A\rangle$ by making $\imath=1_{\bfk}\imath\in \bfk\langle A\rangle$
the identity element and define the product recursively by
\begin{align}\label{eq:inductquasishuffleprod}
au\star_\lambda bv:=a(u\star_\lambda bv)+b(au\star_\lambda v)+\lambda(a\bullet b)(u\star_\lambda v)
\end{align}
for all letters $a,b$ in $A$ and all words $u,v$ in $A^*$.
We remark that the restriction of the commutativity of the product $\bullet$ in the original construction of~\cite{Ho00,HI17} is removed here.

As in~\cite{HI17}, we define the deconcatenation coproduct $\delta:\bfk\langle A\rangle\rightarrow \bfk\langle A\rangle\otimes \bfk\langle A\rangle$ and
counit $\epsilon:\bfk\langle A\rangle\rightarrow\bfk$ by
\begin{align}\label{eq:deltacohp}
    \delta(w):=\sum_{uv=w}u\otimes v\qquad \text{and}\qquad
\epsilon(w):=\begin{cases}
                    1, &\text{if}\  \hbox{$w=\imath$,} \\
                    0, & \hbox{otherwise},
                   \end{cases}
 \tforall w\in A.
 \end{align}
Then $(\bfk\langle A\rangle,\delta,\epsilon)$ is a connected cograded coalgebra with $\bfk\langle A\rangle=\bigoplus_{n\geq0}\bfk\langle A\rangle_n$, where $\bfk\langle A\rangle_n$ is the subspace spanned by the words of length $n$.

Analogous to~\cite{HI17}, we have
\begin{theorem}\label{thm:quasishuffleprodweitlambda}
For any nonzero element $\lambda$ in $\bfk$,
$(\bfk\langle A\rangle,\star_{\lambda},\delta)$ is a Hopf algebra.
Moreover, when $\lambda\neq 0$, the product $\bullet$ is commutative if and only if the product $\star_\lambda$ is commutative.
\end{theorem}

\begin{proof}
The verification for $(\bfk\langle A\rangle,\star_{\lambda},\delta)$ to be a bialgebra is the same as for the classical case~\cite[Theorem 2.1]{HI17} and~\cite[Theorem 3.1]{Ho00}.
Since $\bfk\langle A\rangle_0=\bfk\imath$, the bialgebra $(\bfk\langle A\rangle,\star_{\lambda},\delta)$ is a connected
cograded coalgebra, and hence is a Hopf algebra~\cite{GG19,Man08}.

For the second statement, if $\bullet$ is commutative, then so is $\star_{\lambda}$ as in the classical case of weight $\lambda=1$~\cite{Ho00}.
Conversely, by Eq.~\eqref{eq:inductquasishuffleprod}, for any letters $a,b$,
we have
$$\lambda(a\bullet b-b\bullet a)=a\star_\lambda b-b\star_\lambda a,$$
so the commutativity of $\star_\lambda$ implies the commutativity of $\bullet$.
\end{proof}

If $\lambda=0$ or the product $\bullet$ is identically zero, then $\star_\lambda$ specializes to the usual shuffle product $\shap$ on $\bfk\langle A\rangle$.
If $\lambda$ equals to $1$ and $-1$ respectively, then $\star_\lambda$ coincides with the quasi-shuffle products $*$ and $\star$ in~\cite{HI17,IKOO11} dictating the products of multiple zeta values and multiple zeta-star values, respectively.
Following~\cite{HI17},
we will  call $\star_\lambda$ the \name{quasi-shuffle product of weight $\lambda$ with respect to $\bullet$} and call $(\bfk\langle A\rangle,\star_\lambda)$
the \name{quasi-shuffle algebra of weight $\lambda$}.

We next apply this quasi-shuffle product to signed permutations.

Let $A$ be an alphabet set and fix a subset $Y$ of $A$. Let $\chi_Y$ be the characteristic function of $Y$. Define a product $\bullet_Y$ on $\bfk A$ by setting
\begin{align*}
a\bullet b:=a\bullet_Y b:=\chi_Y(a)\chi_Y(b)a = \begin{cases} a, &\text{if } a, b\in Y,\\
0 & \text{otherwise},
\end{cases}
\quad \text{for all } a, b\in A,
\end{align*}
and applying bilinearity. By a direct check of the associativity of $\bullet_Y$, we obtain

\begin{lemma}\label{lem:circ_YonkA}
Let $A$ be a set and $Y$ a subset of $A$. Then
$(\bfk A,\bullet_Y)$ is an associative $\bfk$-algebra.
\end{lemma}

Now fix a $\lambda\in \bfk$, we define the  shifted quasi-shuffle product $\overline{\star}_\lambda$ of weight $\lambda$ on $\HSym$ with $\imath$ being the identity element.
Take $A$ to be the set $-\PP\cup {\PP}$, and take $Y$ to be the subset $-\PP$.
Then the product $\bullet_Y$ becomes
\begin{align}\label{eq:diamondsp1}
a\bullet b:=\begin{cases}
a,&\text{if}\ a<0, b<0,\\
0,&{\rm otherwise},
\end{cases}
\end{align}
where $a,b$ are nonzero integers. By Lemma \ref{lem:circ_YonkA}, the product $\bullet$ is associative.

For $\sigma\in \mathfrak{B}_m$ and $\tau\in \mathfrak{B}_n$,
define the \name{shifted quasi-shuffle product} $\overline{\star}_\lambda$ of weight $\lambda$ with respect to $\bullet$ by
\begin{align}\label{eq:defnsigmacdotlambdatau}
\sigma \overline{\star}_\lambda \tau:=\st(\sigma \star_{\lambda} \tau[m]),
\end{align}
where $\star_{\lambda}$ is the quasi-shuffle product of weight $\lambda$ in Eq.~\eqref{eq:inductquasishuffleprod}, $\tau[m]$ is the word on the alphabet set $[{m^{-}+n^{-}},m^{-}+1^{-}] \cup[m+1,m+n]$
obtained by replacing each $i\in\mathbb{P}$ in $\tau$ by $i+m$ and each ${i}^{-}\in -\PP$ by ${i^{-}+m^{-}}$, respectively.
For example,
\begin{align}\label{example12st21}
1{2}^{-}\overline{\star}_\lambda2{1}^{-}=&\st(1{2}^{-}\star_{\lambda}4{3}^{-})\notag\\
=&1{2}^{-}4{3}^{-}+14{2}^{-}{3}^{-}+14{3}^{-}{2}^{-}+41{2}^{-}{3}^{-}+41{3}^{-}{2}^{-}+4{3}^{-}1{2}^{-}+\lambda13{2}^{-}+\lambda31{2}^{-}.
\end{align}

\begin{remark}\label{rem:prodstarstwpshiftlarg}
Let $\sigma\in \mathfrak{B}_m$ and $\tau\in \mathfrak{B}_n$. Then, for any positive integer $k\geq m$, the words $\sigma$ and $\tau[k]$ have no common letters,
and it is straightforward to check that $\st(\sigma \star_{\lambda} \tau[m])=\st(\sigma \star_{\lambda} \tau[k])$.
This observation enables us to define the product $\overline{\star}_\lambda$ by
 \begin{align}\label{eq:defncdotweakperm-1}
\sigma \overline{\star}_\lambda \tau=\st(\sigma \star_{\lambda} \tau[k])
\end{align}
for any positive integer $k\geq m$.
 \end{remark}

\begin{prop}
For any $\lambda\in\bfk$,
the $\bfk$-linear space $\HSym$ with the shifted quasi-shuffle product $\overline{\star}_\lambda$
is a unitary $\bfk$-algebra.
\end{prop}
\begin{proof}
It suffices to show that $\overline{\star}_\lambda$ is associative. To this end, it is enough to show that
$(\pi\overline{\star}_\lambda\sigma)\overline{\star}_\lambda\tau=\pi\overline{\star}_\lambda(\sigma\overline{\star}_\lambda\tau)$
for all signed permutations $\pi\in \mathfrak{B}_m,\sigma\in \mathfrak{B}_n,\tau\in \mathfrak{B}_p$.
By Eq.~\eqref{eq:defnsigmacdotlambdatau}, we have
$$(\pi\overline{\star}_\lambda\sigma)\overline{\star}_\lambda\tau=\st(\pi\star_\lambda\sigma[m])\overline{\star}_\lambda\tau.$$ Note that the signed permutations in $\st(\pi\star_\lambda\sigma[m])$
have length at most $m+n$. So by Remark \ref{rem:prodstarstwpshiftlarg}, we have
$$(\pi\overline{\star}_\lambda\sigma)\overline{\star}_\lambda\tau=\st(\st(\pi\star_{\lambda}\sigma[m])\star_{\lambda}\tau[m+n])=\st((\pi\star_{\lambda}\sigma[m])\star_{\lambda}\tau[m+n]).$$
On the other hand, we have
\begin{align*}
\pi\overline{\star}_\lambda(\sigma\overline{\star}_\lambda\tau)=&\pi\overline{\star}_\lambda\st(\sigma\star_{\lambda}\tau[n])=\st(\pi\star_{\lambda}\st(\sigma\star_{\lambda}\tau[n])[m])\\
=&\st(\pi\star_{\lambda}\st(\sigma[m]\star_{\lambda}\tau[m+n]))=\st(\pi\star_{\lambda}(\sigma[m]\star_{\lambda}\tau[m+n])),
\end{align*}
and the associativity of $\overline{\star}_\lambda$ follows from that of $\star_\lambda$ guaranteed by Theorem \ref{thm:quasishuffleprodweitlambda}.
\end{proof}

Define the shifted deconcatenation coproduct $\Delta$ as in Eq.~(\mref{eq:delta}):
\begin{equation}\label{eq:Delta(sigma)WG}
\Delta=(\st\otimes\st)\delta:\HSym\longrightarrow \HSym\otimes\HSym, \quad \sigma\mapsto \sum_{p=0}^m\st(\sigma_1\cdots\sigma_p)\otimes\st(\sigma_{p+1}\cdots\sigma_m)
\end{equation}
for any $\sigma=\sigma_1\sigma_2\cdots\sigma_m\in \mathfrak{B}_m$. For example,
$$
\Delta(3{2}^{-}14{5}^{-})=\imath\otimes 3{2}^{-}14{5}^{-}+1\otimes {2}^{-}13{4}^{-}+2{1}^{-}\otimes 12{3}^{-}+3{2}^{-}1\otimes 1{2}^{-}+3{2}^{-}14\otimes{1}^{-}+3{2}^{-}14{5}^{-}\otimes \imath.
$$
This coproduct is obviously coassociative. Also define a counit
\begin{align}\label{HSymepde}
\epsilon:\HSym\rightarrow \bfk, \quad \epsilon(\sigma):=\begin{cases}
                    1, &\text{if}\  \hbox{$\sigma=\imath$,} \\
                    0, & \text{otherwise.}
                   \end{cases}
\end{align}

\begin{theorem}\label{thm:WSisahopfalg}
For any $\lambda\in\bfk$, the quadruple $(\HSym, \overline{\star}_{\lambda},\Delta,\epsilon)$ is a Hopf algebra, called the \name{Hopf algebra of signed permutations} with weight $\lambda$. It contains the Malvenuto-Reutenauer Hopf algebra $\SSym$ as a Hopf subalgebra.

\end{theorem}

\begin{proof}
We first show that $(\HSym, \overline{\star}_{\lambda},\Delta)$ is a bialgebra. It suffices to show that $\Delta$ and $\epsilon$ are $\overline{\star}_{\lambda}$-preserving.
This is trivial for $\epsilon$. To show the statement for $\Delta$, we first observe that for any word $\pi$ on $-\PP\cup {\PP}$, we have
$\Delta(\st(\pi))=\Delta(\pi)$. By Theorem \ref{thm:quasishuffleprodweitlambda}, $\delta$ is a $\star_{\lambda}$-homomorphism. Then for any $\sigma\in \mathfrak{B}_m$ and $\tau\in \mathfrak{B}_n$,
we have
\begin{eqnarray*}
&&
\Delta(\sigma\overline{\star}_{\lambda} \tau)\\
&=&\Delta(\st(\sigma\star_{\lambda} \tau[m]))\\
&=&\Delta(\sigma\star_{\lambda} \tau[m])\\
&=&(\st\otimes\st)\delta(\sigma\star_{\lambda} \tau[m])\\
&=&(\st\otimes\st)(\delta(\sigma)\star_{\lambda} \delta(\tau[m]))\hspace{62mm}\text{(by Theorem~\ref{thm:quasishuffleprodweitlambda})}\\
&=& \sum_{p=0}^m\sum_{q=0}^n\st(\sigma_1\cdots \sigma_p\star_{\lambda}(\tau_1\cdots \tau_q)[m])\otimes \st(\sigma_{p+1}\cdots \sigma_m\star_{\lambda}(\tau_{q+1}\cdots \tau_n)[m]) \quad\text{(by Eq. \eqref{eq:deltacohp})}\\
&=& \sum_{p=0}^m\sum_{q=0}^n\st(\st(\sigma_1\cdots \sigma_p)\star_{\lambda}\st(\tau_1\cdots \tau_q)[m])\otimes
\st(\st(\sigma_{p+1}\cdots \sigma_m)\star_{\lambda}\st(\tau_{q+1}\cdots \tau_n)[m]) \\
&&\hspace{105mm}\text{(by the definition of}\ \st)\\
&=& \sum_{p=0}^m\sum_{q=0}^n(\st(\sigma_1\cdots \sigma_p)\overline{\star}_{\lambda}\st(\tau_1\cdots \tau_q))\otimes
(\st(\sigma_{p+1}\cdots \sigma_m)\overline{\star}_{\lambda}\st(\tau_{q+1}\cdots \tau_n))\quad \text{(by Eq.~\eqref{eq:defncdotweakperm-1})}\\
&=&(\st\otimes\st)(\delta(\sigma))\overline{\star}_{\lambda} (\st\otimes\st)(\delta(\tau))\hspace{55mm}\text{(by Eq. \eqref{eq:deltacohp})}\\
&=&\Delta(\sigma)\overline{\star}_{\lambda} \Delta(\tau).
\end{eqnarray*}
This shows that $\Delta$ is a $\overline{\star}_{\lambda}$-homomorphism. Hence $\HSym$ is a bialgebra.

Let $\HSym_n=\bfk \mathfrak{B}_n, n\geq 0$. Then $\HSym=\bigoplus_{n\geq0}\HSym_n$ is a cograded coalgebra, that is,
$$
\Delta(\HSym_n)\subseteq\bigoplus_{p+q=n}\HSym_p\otimes \HSym_q \ \text{for}\ n\geq0,\ \bigoplus_{n\geq1}\HSym_n\subseteq \ker\varepsilon.
$$
Note that $\HSym_0=\bfk \imath$, so $(\HSym, \overline{\star}_{\lambda},\Delta)$ is a bialgebra such that $(\HSym, \Delta,\varepsilon)$ is a connected cograded coalgebra.
By~\cite{GG19,Man08},  $\HSym$ is a Hopf algebra.

The restriction of $\overline{\star}_\lambda$ to permutations is the shifted shuffle product and the coproduct $\Delta$ on signed permutations extends the one on permutations, proving the second statement.
\end{proof}

\section{Relationship among the Hopf algebras}
\label{sec:BaP}
In this section, we provide surjective Hopf algebra homomorphisms $\varphi_2$ and $\mathcal{D}_2$ from the Hopf algebra $\HSym$ of weak permutations to the Hopf algebra $\RQSym$ of
weak quasi-symmetric functions and the Hopf algebra $\SSym$ of permutations respectively, so that the commutativity of the diagram \eqref{eq:diag2} holds.
We begin by providing background on quasi-symmetric and weak quasi-symmetric functions. See~\cite{BB05,LMW13,Mac95,Sta12,YGT} for further details.

\subsection{Weak compositions and regularized compositions}

Let $\widetilde{\NN}:=\mathbb{N}\cup \{\varepsilon\}$ be the additive monoid obtained from the additive monoid $\NN$ by adjoining a new element $\varepsilon$ such that
$$0+\varepsilon=\varepsilon+0=\varepsilon+\varepsilon=\varepsilon,\ n+\varepsilon=\varepsilon+n=n\quad \text{for all } n\geq1.$$
For convenience we extend the natural order on $\NN$ to $\widetilde{\NN}$ by defining $0<\varepsilon<1$.

Given an element $n\in\widetilde{\mathbb{N}}$, a \name{weak $\widetilde{\mathbb{N}}$-composition} $\alpha=(\alpha_1,\alpha_2,\cdots,\alpha_k)$ of $n$ is an ordered finite sequence of elements of $\widetilde{\mathbb{N}}$ whose sum is $n$.
We call $\alpha_i$ the \name{parts} of $\alpha$, $n$ the \name{weight} of $\alpha$,
denoted by $|\alpha|$, and $k$ the \name{length} of $\alpha$, denoted by $\ell(\alpha)$.
In particular, we call $\alpha$ a \name{composition} if all of its parts are positive, a \name{weak composition} if all of its parts are nonnegative, a \name{regularized composition} if all of its parts are either positive or $\vep$, regarded as the ``regularization" of a weak composition when the zero parts of the weak composition are replaced by $\vep$.
We write $\alpha\models n$ if $\alpha$ is a composition of weight $n$.

For example, 6 has a composition $(1,2,3)$, a weak composition $(1,0,2,3)$, a regularized composition $(1,2,\varepsilon,3)$ and a weak $\widetilde{\mathbb{N}}$-composition $(1,0,2,\varepsilon,3)$.
For convenience, we denote by $\emptyset$ the unique weak $\widetilde{\mathbb{N}}$-composition, called the \name{empty composition}, whose weight and length are designated to be $0$.
Given $n$ in $\widetilde{\mathbb{N}}$, let $\mathcal{C}(n)$ and $\widetilde{\mathcal{C}}(n)$ denote the set of compositions and regularized compositions of $n$, respectively.
We write $\mathcal{C}$, $\mathcal{WC}$ and $\widetilde{\mathcal{C}}$ for the sets of compositions, weak compositions and
regularized compositions, respectively. Clearly, we have $\mathcal{C}=\mathcal{WC}\cap \widetilde{\mathcal{C}}$.

Replacing a weak composition $\alpha$ by its regularized composition noted above starts with the map
\begin{align*}
\theta: \NN\longrightarrow \widetilde{\NN}\backslash \{0\}, \quad
n\mapsto
\left\{
\begin{array}{ll}
\vep, & \text{if}\  n=0, \\ n, & \text{otherwise},
\end{array}
\right.
\end{align*}
and is given by the natural bijection
\begin{align*}
\hat{\theta}:\mathcal{WC} \longrightarrow  \widetilde{\mathcal{C}},\quad
\alpha \mapsto
\left\{
\begin{array}{ll}
\emptyset, & \text{if}\ \alpha=\emptyset,\\
(\theta(\alpha_1),\cdots, \theta(\alpha_k)), &\text{if}\  \alpha=(\alpha_1,\cdots,\alpha_k)\in\mathcal{WC}, k\geq 1.
\end{array}
\right.
\end{align*}
This enables us to define weak quasi-symmetric functions in Section~\ref{subsection:defnHopfalgQWQSs}.

We write $\varepsilon^i$ for a string of $i$ parts of $\varepsilon$.
Then any regularized composition $\alpha$  can be expressed uniquely in the form
\begin{align*}
\alpha=(\varepsilon^{i_1},s_1,\varepsilon^{i_2},s_2,\cdots,\varepsilon^{i_k}, s_k,\varepsilon^{i_{k+1}}),
\end{align*}
where $k\in \NN$, $i_p\in \NN$ for $p\in[k+1]$ and $s_q\in \PP$ for $q\in[k]$.
Denote by $\bar{\alpha}$ the composition obtaining from the regularized composition $\alpha$ by omitting its $\varepsilon$ parts.
Then $\bar{\alpha}=(s_1,s_2,\cdots,s_k)$. Further, for the length and weight of $\alpha$, we have
$$\ell(\alpha)=k+i_1+i_2+\cdots+i_{k+1},\ |\alpha|=\left\{\begin{array}{ll} 0, & \text{if}\ k=i_1=0,\\
\varepsilon, &\text{if}\ k=0, i_1\ge1,\\
s_1+s_2+\cdots+s_k, &\text{if}\  k>0. \end{array}\right.$$
We call $\ell_{\varepsilon}(\alpha):=i_1+i_2+\cdots+i_{k+1}$ the \name{$\varepsilon$-length} of $\alpha$ and  $||\alpha||:=|\alpha|+\ell_{\varepsilon}(\alpha)$
the \name{total weight} of $\alpha$.
Clearly, $|\alpha|=||\alpha||$ if and only if $\alpha$ is a composition.
The \name{descent set} of $\alpha$ is the subset of $[||\alpha||]$ defined by
\begin{align*}
    D({\alpha}):=\left\{\left. b_q:=\sum_{j=1}^q(i_j+s_j)\right| q\in [k]\right\}.
\end{align*}
Note that  $||\alpha||\in D(\alpha)$ if $i_{k+1}=0$, and $||\alpha||\notin D(\alpha)$ otherwise.
Thus, the integer $||\alpha||=|\alpha|$ is in the descent set if $\alpha$ is a composition, which is different from the usual notation in the literature~\cite{Sta12}.
This is because we will need to use the total weight $||\alpha||$ to determine whether the last part of $\alpha$ is $\varepsilon$ or a positive integer, when $\alpha$ is a regularized composition.

Any composition $\alpha=(\alpha_1,\alpha_2,\cdots,\alpha_k)$ of $n$ naturally corresponds to its descent set
$D(\alpha)=\{\alpha_1,\alpha_1+\alpha_2,\cdots,\alpha_1+\alpha_2+\cdots+\alpha_{k-1}, n\}$.
Note that this is an invertible procedure; that is, if $S=\{a_1<a_2<\cdots<a_k<n\}$ is a subset of $[n]$, then the composition ${\rm comp}(S):=(a_1,a_2-a_1,\cdots,n-a_k)$ is precisely the composition
such that $D({\rm comp}(S))=S$.

Let $\alpha$ and $\beta$ be regularized compositions of weight $n$. By grouping adjacent positive parts into blocks, we might write
$\alpha=(\vep^{i_1},\alpha_1,\vep^{i_2},\alpha_2,\cdots,\vep^{i_k}, \alpha_k, \vep^{i_{k+1}})$ and
$\beta=(\vep^{j_1},\beta_1,\vep^{j_2},\beta_2,\cdots,\vep^{j_\ell}, \beta_\ell, \vep^{j_{\ell+1}})$, where $i_p, j_q\in \NN$ for $p\in [k+1], q\in [\ell+1]$ and $\alpha_p, \beta_q$ are compositions for $p\in [k], q\in [\ell]$.
We say that $\alpha$ is a \name{refinement} of $\beta$ (or equivalently $\beta$ is a \name{coarsening} of $\alpha$),
denoted by $\alpha\preceq\beta$, if $k=\ell$, $D(\alpha_p)\supseteq D(\beta_p)$, $i_p\leq j_p$ for all $p\in[k]$ and either $i_{k+1}=j_{k+1}=0$ or $1\leq i_{k+1}\leq j_{k+1}$.
For example, $(1,2,\varepsilon^2,1,3,2,\varepsilon)\preceq(3,\varepsilon^2,1,\varepsilon,5,\varepsilon^3)$, but
$(1,2,\varepsilon^2,1,3,2)$ and $(3,\varepsilon^2,1,\varepsilon,5,\varepsilon^3)$ are not comparable.
Restricting $\preceq$ to the subset $\mathcal{C}(n)$, we obtain the refinement order on compositions.

Given a regularized composition $\alpha$, the \name{ reversal} of $\alpha$, denoted by $\alpha^{r}$, is obtained by writing the entries of $\alpha$ in the reverse order.
For regularized compositions $\alpha=(\alpha_1,\alpha_2,\cdots,\alpha_k)$ and $\beta=(\beta_1,\beta_2,\cdots,\beta_l)$,
the \name{concatenation} of $\alpha$ and $\beta$ is
$\alpha\cdot \beta:=(\alpha_1,\cdots,\alpha_k,\beta_1,\cdots,\beta_l)$,
while if (and only if) $\alpha_k\in\mathbb{P}$ and $\beta_1\in\mathbb{P}$, then the \name{near concatenation} is
$\alpha\odot \beta:=(\alpha_1,\cdots,\alpha_k+\beta_1,\cdots,\beta_l)$.
Let $J=(j_1,j_2,\cdots,j_l)$ be a composition of $k$.
The regularized composition $J[\alpha]$ of $\alpha$ is defined by
\begin{align*}
J[\alpha]:=(\alpha_1+\cdots+\alpha_{j_1},\alpha_{j_1+1}+\cdots+\alpha_{j_1+j_2},\cdots,\alpha_{j_1+j_2+\cdots+j_{l-1}+1}+\cdots+\alpha_{k}).
\end{align*}
For example, take $\alpha=(3,1,\varepsilon)$, $\beta=(2,\varepsilon)$ and $J=(1,2)$. Then $\alpha^r=(\varepsilon,1,3)$,  $J[\alpha]=(3,1)$,
$\alpha\cdot\beta=(3,1,\varepsilon,2,\varepsilon)$ and $\alpha^r\odot\beta=(\varepsilon,1,5,\varepsilon)$, while $\alpha\odot\beta$ is not defined.

\subsection{Weak quasi-symmetric functions} \label{subsection:defnHopfalgQWQSs}
This subsection recalls the Hopf algebras $\RQSym$ of weak quasi-symmetric functions, while refers the reader to~\cite{YGT} for further details.

We begin with generalizing the formal power series algebra.
Let $X=\{x_n\,|\,n\in \PP\}$ be a set of commuting variables.
A formal power series is a (possibly infinite) linear combination of monomials $x_{i_1}^{\alpha_1}x_{i_2}^{\alpha_2}\cdots x_{i_k}^{\alpha_k}$ where $\alpha_1,\alpha_2,\cdots,\alpha_k$ are positive integers, which can be regarded as the locus of the map from $X$ to $\NN$ sending $x_{i_j}$ to $\alpha_j$, $1\leq j\leq k$, and everything else in $X$ to zero.
Our generalization of the formal power series algebra is simply to replace $\NN$ by $\widetilde{\NN}$.

The set of {\bf $\widetilde{\NN}$-valued maps} is defined by
\begin{equation*}
\map{X}{\widetilde{\NN}}:=\left\{\left. f:X\to \widetilde{\NN}\,\right |\, \supp (f) \text{ is finite }\right\},
\end{equation*}
where $\supp(f):=\{x\in X\,|\, f(x)\neq 0\}$ denotes the support of $f$.
The addition on $\widetilde{\NN}$ equips $\map{X}{\widetilde{\NN}}$ with an addition by
$$ (f+g)(x):=f(x)+g(x)\quad  \text{ for all } f, g\in \widetilde{\NN}^X \text{ and } x\in X,$$
making $\map{X}{\widetilde{\NN}}$ into an additive monoid.
Resembling the formal power series, we identify $f\in \map{X}{\widetilde{\NN}}$ with its locus $\{(x,f(x))\,|\,x\in \supp(f)\}$ expressed in the form of a formal product, called an {\bf $\widetilde{\NN}$-exponent monomial}:
\begin{align*}
X^f:=\prod_{x\in X}x^{f(x)}=\prod_{x\in \supp(f)}x^{f(x)},
\end{align*}
with the convention $x^0=1$. So each $\widetilde{\NN}$-exponent monomial can be written as $x_{i_1}^{\alpha_1}x_{i_2}^{\alpha_2}\cdots x_{i_k}^{\alpha_k}$ where
$\alpha=(\alpha_1,\alpha_2,\cdots,\alpha_k)$ is a regularized composition.

With this notation, the (multiplicative) addition on $\map{X}{\widetilde{\NN}}$ is simply
\begin{equation*}
X^fX^g=X^{f+g}\quad \text{ for all } f, g\in \map{X}{\widetilde{\NN}}.
\mlabel{eq:xmap}
\end{equation*}
We then form the semigroup algebra
$\bfk [X]_{\widetilde{\NN}}:=\bfk \map{X}{\widetilde{\NN}}$
consisting of linear combinations of  $\widetilde{\NN}$-exponent monomials, called the algebra of \name{$\widetilde{\NN}$-exponent polynomials}.
Similarly, we can define the $\bfk$-module $\bfk[[X]]_{\widetilde{\NN}}$ consisting of possibly infinite linear combinations of  $\widetilde{\NN}$-exponent monomials,
called {\bf $\widetilde{\NN}$-exponent formal power series}.

Let $\alpha=(\alpha_1,\alpha_2,\cdots,\alpha_k)$ be a regularized composition. We call $|\alpha|$ and $||\alpha||$ the
 \name{degree} and \name{ total degree} of the $\widetilde{\NN}$-exponent monomial  $x_{i_1}^{\alpha_1}x_{i_2}^{\alpha_2}\cdots x_{i_k}^{\alpha_k}$, respectively.
Analogously, the degree (resp. total degree) of an $\widetilde{\NN}$-exponent  power series $f$ is the largest $|\alpha|$  (resp. $||\alpha||$) such that
the coefficient of $x_{i_1}^{\alpha_1}x_{i_2}^{\alpha_2}\cdots x_{i_k}^{\alpha_k}$ in $f$ is not zero.

\begin{defn}\label{def:bquasisymm}
A formal power series $f$ of bounded total degree in $\bfk [[X]]_{\widetilde{\NN}}$  is called a \name{ weak quasi-symmetric function} if, for any regularized composition
$(\alpha_1,\alpha_2,\cdots,\alpha_k)$, the coefficients of
$x_{i_1}^{\alpha_1}x_{i_2}^{\alpha_2}\cdots x_{i_k}^{\alpha_k}$
and $x_{j_1}^{\alpha_1}x_{j_2}^{\alpha_2}\cdots x_{j_k}^{\alpha_k}$ in $f$ are equal for every two increasing sequences $i_1<i_2<\cdots<i_k$ and $j_1<j_2<\cdots<j_k$ of positive integers.
\end{defn}

The set of weak quasi-symmetric functions is a
subalgebra of $\bfk [[X]]_{\widetilde{\NN}}$ denoted by  $\RQSym(X)$, or $\RQSym$ for short, replacing the notion $\QSym_{\widetilde{\NN}}$ used in~\cite{YGT}. The prefix R stands for the ``regularization" $x^\vep$ in place of $x^0$. For this reason, we also call the weak quasi-symmetric functions the {\bf regularized quasi-symmetric functions}.

Rota suggested that Rota-Baxter algebras provide an ideal context for generalizations of symmetric functions~\cite{Ro95}. Motivated by this suggestion, we obtained

\begin{theorem} $($\cite[Theorem~4.4]{YGT}$)$
Let $X=\{x_n\,|\, n\in \PP\}$ be as above and let $y$ be a letter not in $X$. The tensor product algebra $\bfk[y]\otimes \RQSym$ is isomorphic to the free commutative Rota-Baxter algebra on $\{y\}$ constructed in~\cite{G-K1}.
\end{theorem}

For a regularized composition $(\alpha_1,\alpha_2,\cdots,\alpha_\ell)$,
the \name{monomial weak quasi-symmetric function } indexed by $\alpha$ is the formal power series
\begin{align*}
M_{\mathbf{\alpha}}:=\sum_{j_1<j_2<\cdots <j_\ell} x_{j_1}^{\alpha_1}x_{j_2}^{\alpha_2}\cdots x_{j_\ell}^{\alpha_\ell}.
\end{align*}

Write $\alpha=(\varepsilon^{i_1},s_1,\varepsilon^{i_2},s_2,\cdots,\varepsilon^{i_k},s_k, \varepsilon^{i_{k+1}})$,
where $k\in\mathbb{N}$,  $i_p\in\mathbb{N}$ for $p\in[k+1]$ and $s_q\in\mathbb{P}$ for $q\in[k]$.
Let $D(\alpha)=\{b_1,b_2,\cdots,b_{k}\}$ denote the descent set of $\alpha$. So $b_q=\sum_{j=1}^q(i_j+s_j)$ for $q\in[k]$.
Then the \name{fundamental weak quasi-symmetric function}
indexed by $\alpha$ is the formal power series
\begin{align}\label{eqdefn:fundmbasis1}
F_{\alpha}:=&\sum_{{j_1\leq j_2\leq \cdots \leq j_{n} \atop p\in D(\alpha) \Rightarrow j_{p}<j_{p+1}}}\prod_{t=1}^{i_1} x_{j_t}^\varepsilon\cdot\prod_{t=i_1+1}^{b_1} x_{j_t} \cdot
\prod_{t=b_1+1}^{b_1+i_2} x_{j_t}^\varepsilon\cdot\prod_{t=b_1+i_1+1}^{b_2} x_{j_t}\cdots \prod_{t=b_{k-1}+1}^{b_{k-1}+i_{k}} x_{j_t}^\varepsilon\cdot\prod_{t=b_{k-1}+i_{k}+1}^{b_k} x_{j_t}
\cdot \prod_{t=b_{k}+1}^{b_{k}+i_{k+1}} x_{j_t}^\varepsilon\notag\\
=&\sum_{{j_1\leq j_2\leq \cdots \leq j_{n} \atop p\in D(\alpha) \Rightarrow j_{p}<j_{p+1}}}
x_{j_1}^\varepsilon\cdots x_{j_{i_1}}^\varepsilon x_{j_{i_1+1}}\cdots x_{j_{b_1}}
\cdots x_{j_{b_{k-1}+1}}^\vep \cdots x_{j_{b_{k-1}+i_k}}^\varepsilon x_{j_{b_{k-1}+i_k+1}}\cdots x_{j_{b_k}}x_{j_{b_{k}+1}}^\varepsilon\cdots x_{j_{n}}^\varepsilon.
\end{align}
Let $M_{\emptyset}=F_{\emptyset}=1$. Then $\{M_{\alpha}\,|\,\alpha\in \widetilde{\mathcal{C}}\}$ and  $\{F_{\alpha}\,|\,\alpha\in \widetilde{\mathcal{C}}\}$ are linear bases for $\RQSym$.
Also, by~\cite[Propositions 3.1 and 3.2]{YGT}, for any regularized composition $\alpha=(\varepsilon^{i_1},s_1,\cdots,\varepsilon^{i_k},s_k, \varepsilon^{i_{k+1}})$,
we have
\begin{align*}
    F_{\alpha}= \sum_{\beta\preceq\alpha}c_{\alpha,\beta} M_\beta \quad\ \text{and}  \quad M_\alpha=\sum_{\beta\preceq\alpha} (-1)^{\ell(\beta)-\ell(\alpha)}c_{\alpha,\beta} F_\beta,
\end{align*}
where $c_{\alpha,\beta}=\binom{i_1}{j_1}\cdots\binom{i_k}{j_k}\binom{i_{k+1}-1}{j_{k+1}-1}$ for any regularized composition $\beta=(\varepsilon^{j_1},\beta_1,\cdots,\varepsilon^{j_k},\beta_k,\varepsilon^{j_{k+1}})$ such that $\beta\preceq\alpha$,
with the convention that $\binom{-1}{-1}=1$.

Let $\RQSym_{n}$ denote the linear span of $\{M_{\alpha}\,|\,\alpha\in \widetilde{\mathcal{C}}(n)\}$ for $n\in \widetilde{\NN}$.
It follows from~\cite[Propositions 2.5]{YGT} that the vector space $\RQSym=\bigoplus_{n\in\widetilde{\NN}}\RQSym_{n}$ of weak
quasi-symmetric functions has the structure
of an $\widetilde{\NN}$-graded Hopf algebra.
The product is the ordinary multiplication of $\widetilde{\NN}$-exponent power series.
With respect to the monomial weak quasi-symmetric functions, the product is given by the quasi-shuffle product $*$~\cite{Ho00} of regularized compositions, which can be equivalently defined by the stuffle product~\cite{BBBL} or mixable shuffle product~\cite{G-K1}. See~\cite{Gub12}. More precisely, given regularized compositions
$\alpha=(\alpha_1,\alpha_2,\cdots,\alpha_k)$ and $\beta=(\beta_1,\beta_2,\cdots,\beta_l)$. Let $\alpha'=(\alpha_2,\cdots,\alpha_k)$ and $\beta'=(\beta_2,\cdots,\beta_l)$.
Define
\begin{align*}
    \alpha*\emptyset=\emptyset*\alpha:=\alpha,\quad \alpha*\beta:=(\alpha_1,\alpha'*\beta)+(\beta_1,\alpha*\beta')+(\alpha_1+\beta_1,\alpha'*\beta').
\end{align*}
For convenience, we define a $\bfk$-bilinear pair $\langle\cdot,\cdot\rangle$ on $\bfk\widetilde{C}$ satisfying $\langle\alpha,\beta\rangle=\delta_{\alpha,\beta}$.
Then
\begin{align*}
    M_{\alpha}M_{\beta}=\sum_{\gamma\in\widetilde{C}}\langle\gamma,\alpha*\beta\rangle M_{\gamma}.
\end{align*}
Thus, we will write $M_{\alpha}M_{\beta}=M_{\alpha*\beta}$, for simplicity.
The coproduct, counit and antipode are defined by
\begin{align}\label{deepSant}
    \Delta_{R}(M_{\alpha}):=\sum_{\alpha=\beta\cdot\gamma}M_{\beta}\otimes M_{\gamma}, \qquad \epsilon_{R}(M_\alpha):=\delta_{\alpha,\emptyset},\qquad
    S_{R}(M_\alpha):=(-1)^{\ell(\alpha)}\sum_{J\models \ell(\alpha)}M_{J[\alpha^r]}.
\end{align}
The coproduct formula of a fundamental weak quasi-symmetric function is given in~\cite{Li18} by
\begin{align}\label{eq:antipodeforwqsymfundam}
   \Delta_{R}(F_{\alpha}):=\sum_{\alpha=\beta\cdot\gamma}F_{\beta}\otimes F_{\gamma}+\sum_{\alpha=\beta\odot\gamma}F_{\beta}\otimes F_{\gamma}.
\end{align}
When $\alpha$'s are taken to be compositions, we obtain the Hopf subalgebra $\QSym$ of quasi-symmetric functions as usual.

\subsection{Commutativity of the square}
\label{Sec:commdiag}
We now extend the commutative diagram of Hopf algebras in Eq.~\eqref{eq:diag1} to the one in Eq.~\eqref{eq:diag2}, duplicated here for the reader's convenience.
\begin{equation}
\begin{split}
\xymatrix{
\HSym \atop \text{(signed permutations)} \ar@{-->>}[rr]^{\varphi_2}\ar@{-->>}[d]_{\mathcal{D}_2}^{\text{weak descents \&} \atop \text{signed } P-\text{partitions}}
&&\SSym \atop \text{(permutations)} \ar@{->>}[d]_{\mathcal{D}_1}^{\text{descents \&}\atop P\text{-partitions}} &&\ \NSym \ar@{>->}[ll]  \ar@{->>}[d] \\
  \RQSym \atop  \text{(weak compositions)} \ar@{->>}[rr]^{\varphi_1}
&&\QSym \atop  \text{(compositions)} && \ \text{Sym} \ar@{>->}[ll]
}
\mlabel{eq:diag2b}
\end{split}
\end{equation}
Here the Hopf algebra $\HSym$ is assume to have weight $-1$. Thus, the multiplication on $\HSym$ defined by Eq.~\eqref{eq:defnsigmacdotlambdatau}
becomes $\sigma \overline{\star}_{-1} \tau=\st(\sigma \star_{-1} \tau[m])$.

A signed permutation $\pi\in \mathfrak{B}_n$  has a \name{descent} at position $i\in[0,n-1]$ if $\pi_i>\pi_{i+1}$. Define the
 \name{descent set}  $\pi$ by
$D(\pi):=\{i\in[0,n-1]|\pi_i>\pi_{i+1}\}.$
The descents of signed permutations received significant attention over the past decades, see, for example~\cite{BB92,MR95,PK05,Sol76}.
In order to give the desired Hopf algebra homomorphism $\mathcal{D}_2$, we need a new notion as follows.
The \name{weak descent set}, denoted $WD(\pi)$,  of a signed permutation $\pi\in \mathfrak{B}_n$ is defined by
\begin{align}\label{eq:defn:descentofpi}
WD(\pi):=\begin{cases}
\{i\in [n-1]\,|\,\pi_{i}>\text{max}(0,\pi_{i+1})\},&\text{if}\  \pi_{n}< 0,\\
\{i\in [n-1]\,|\,\pi_{i}>\text{max}(0,\pi_{i+1})\}\cup \{n\},&\text{otherwise}.
\end{cases}
\end{align}
In particular, if $\pi\in\mathfrak{S}_n$ is a permutation, then $WD(\pi)=D(\pi)\cup\{n\}$.
For any permutation $\pi\in\mathfrak{S}_n$, define its \name{descent composition}, denoted $\rm{comp}(\pi)$, to be $\rm{comp}(WD(\pi))$.

As shown in~\cite[Th\'eor\`ems 5.12, 5.13, and 5.18]{Mal94}, there exists a Hopf algebra surjection
\begin{align}\label{eq:D1defn}
    \mathcal{D}_1:\SSym \rightarrow \QSym,\quad
                  \pi \mapsto F_{\text{comp}(\pi)}.
\end{align}
By~\cite[Theorem 3.8]{YGT}, there is a Hopf algebra surjection $\varphi_1:\RQSym\rightarrow \QSym$ defined by
\begin{align*}
    \varphi_1(M_{\alpha}):=\begin{cases}
                    (-1)^{\ell_\varepsilon(\alpha)}M_{\overline{\alpha}}, &\hbox{if the first part of $\alpha$ is a positive integer,} \\
                    \delta_{\alpha,\emptyset}, & \hbox{otherwise;}
                   \end{cases}
\end{align*}
where  $\delta_{\alpha,\emptyset}$ denotes the Kronecker delta. Equivalently, by~\cite[Section 4.2]{Li18},
\begin{align}\label{eq:varphi(Falpha)}
    \varphi_1(F_{\alpha}):=\begin{cases}
                    (-1)^{j}F_{\overline{\alpha}}, &\hbox{if $\alpha=(\varepsilon^i,\overline{\alpha},\varepsilon^j)$ for some $i\in\mathbb{N}$, $j\in\{0,1\}$ and $\overline{\alpha}\neq\emptyset$, } \\
                    \delta_{\alpha,\emptyset}, & \hbox{otherwise.}
                   \end{cases}
\end{align}

For a signed permutation $\pi$, we denote by $\overline{\pi}$ the word obtained from $\pi$ by omitting its negative parts.
For example, if $\pi=53^{-}246^{-}1^{-}$, then $\overline{\pi}=524$. Thus, a signed permutation $\pi$ satisfies $\pi=(\pi_1,\cdots,\pi_i,\overline{\pi},\pi_{n-j+1},\cdots,\pi_n)$ for some
nonnegative integers $i$ and $j$
if and only if $\pi_k<0$ for $k\in[i]\cup[n-j+1,n]$ and $\pi_k>0$ for $k\in[i+1,n-j]$.
Then we obtain a linear map $\varphi_2:\HSym\rightarrow \SSym$ by taking
\begin{align*}
    \varphi_2(\pi):=\begin{cases}
                    (-1)^{j}{\st(\overline{\pi})}, & \hbox{if $\pi=(\pi_1,\cdots,\pi_i,\overline{\pi},\pi_{n-j+1},\cdots,\pi_n)$ for some $i\in\mathbb{N}$,
                    $j\in\{0,1\}$ and $\overline{\pi}\neq\imath$, } \\
                    \delta_{\pi,\imath}, & \hbox{otherwise.}
                   \end{cases}
\end{align*}

Observe that the map $\varphi_2$ defined above is independent of the values of the  negative integers in $\pi=\pi_1\cdots \pi_n$. Thus to study $\varphi_2$, no information is lost in substituting $\smallsmile$ for those $\pi_i$ with $\pi_i<0$, which we will adapt for notational clarity. For example, for the previous $\pi$ we have $\pi=(5,\smallsmile,24,\smallsmile^2)$.
In general, any signed permutation $\pi\in \mathfrak{B}_n$ will be written as
$$ \pi=({\smallsmile}^{i_1},\pi_{i_1+1}\cdots \pi_{j_1},{\smallsmile}^{i_2},\pi_{j_1+i_2+1}\cdots \pi_{j_2},\cdots, {\smallsmile}^{i_k},\pi_{j_{k-1}+i_k+1}\cdots \pi_{j_k},{\smallsmile}^{i_{k+1}})$$
for clarity, where
$i_p\in\mathbb{N}$, $p\in[k+1]$, $j_{q-1}+i_q+1\leq j_{q}$, and $\pi_{l}\in\PP$ for $j_{q-1}+i_q+1\leq l\leq j_{q}$, $q\in[k]$ with the convention $j_{0}=0$.

With this notation, the map $\varphi_2:\HSym\rightarrow \SSym$ is simply

\begin{align}\label{eq:varphisigmawgtossym}
    \varphi_2(\pi):=\begin{cases}
                    (-1)^{j}{\st(\overline{\pi})}, & \hbox{if $\pi=({\smallsmile}^i,\overline{\pi},{\smallsmile}^j)$ for some $i\in\mathbb{N}$,
                    $j\in\{0,1\}$ and $\overline{\pi}\neq\imath$, } \\
                    \delta_{\pi,\imath}, & \hbox{otherwise.}
                   \end{cases}
\end{align}

Evidently $\varphi_2$ is surjective since it is the identity on permutations.
Also note that $\varphi_2(\pi)$ is zero as long as $\pi$ has multiple blocks of positive parts (separated by negative parts).

With this notation, the
\name{regularized composition} of $\pi$ is defined by
\begin{align}\label{eq:lemwcomppi}
{\wcomp}(\pi):= (\varepsilon^{i_1},{\rm{comp}}(\st(\pi_{i_1+1}\cdots \pi_{j_1})),\varepsilon^{i_2},&{\rm{comp}}(\st(\pi_{j_1+i_2+1}\cdots \pi_{j_2})),
    \cdots,\notag\\
     &\varepsilon^{i_k},{\rm{comp}}(\st(\pi_{j_{k-1}+i_k+1}\cdots \pi_{j_k})),\varepsilon^{i_{k+1}}).
\end{align}
For example, $\wcomp(5,\smallsmile,243,\smallsmile^2)=(1,\varepsilon,2,1,\varepsilon^2)$.
We then further define the $\bfk$-linear map
\begin{align}\label{eq:defnD2}
\mathcal{D}_2:\HSym\rightarrow \RQSym, \quad  \pi\mapsto F_{{\wcomp}(\pi)}.
\end{align}

We now arrive at the main result of this paper about the commutativity of the left square in diagram~\eqref{eq:diag2b}.

\begin{theorem}
The linear maps $\varphi_2$ and $\mathcal{D}_2$ are Hopf algebra homomorphisms such that $\mathcal{D}_1\varphi_2=\varphi_1\mathcal{D}_2$.
\label{thm:comm}
\end{theorem}
\begin{proof}
We first verify the commutativity of the left square of the diagram (\mref{eq:diag2b}) as linear structures. The proofs that $\varphi_2$ and $\mathcal{D}_2$ are Hopf algebra homomorphisms need more preparations and will be provided in Section~\mref{sec:phi2} and Section~\mref{sec:d2} respectively.

By Eqs.~\eqref{eq:D1defn} and \eqref{eq:varphisigmawgtossym}, for any signed permutation $\pi$ of $[n]$, we have
\begin{align*}
 \mathcal{D}_1\varphi_2(\pi)=&
 \begin{cases}
                    (-1)^{j}\mathcal{D}_1({\st(\overline{\pi})}), & \hbox{if $\pi=({\smallsmile}^i,\overline{\pi},{\smallsmile}^j)$ for some $i\in\mathbb{N}$, $j\in\{0,1\}$ and $\overline{\pi}\neq\imath$,} \\
                    \delta_{\pi,\imath}, & \hbox{otherwise.}
 \end{cases}\\
 =&\begin{cases}
                    (-1)^{j}F_{\text{comp}({\st(\overline{\pi})})}, & \hbox{if $\pi=({\smallsmile}^i,\overline{\pi},{\smallsmile}^j)$ for some $i\in\mathbb{N}$, $j\in\{0,1\}$ and $\overline{\pi}\neq\imath$,} \\
                    \delta_{\pi,\imath}, & \hbox{otherwise.}
 \end{cases}
\end{align*}
On the other hand,  it follows from Eqs.~\eqref{eq:varphi(Falpha)} and \eqref{eq:defnD2} that
\begin{align*}
    \varphi_1\mathcal{D}_2(\pi)=&\varphi_1(F_{{\wcomp}(\pi)})
    =\begin{cases}
                    (-1)^{j}F_{\overline{\alpha}}, & \hbox{if $\alpha=(\varepsilon^i,\overline{\alpha},\varepsilon^j)$ for some $i\in\mathbb{N}$, $j\in\{0,1\}$ and $\overline{\alpha}\neq\emptyset$, } \\
                    \delta_{\alpha,\imath}, & \hbox{otherwise,}
                   \end{cases}
\end{align*}
where $\alpha={\wcomp}(\pi)$.
Note that ${\wcomp}(\pi)=(\varepsilon^i,\overline{{\wcomp}(\pi)},\varepsilon^j)$ is equivalent to $\pi=({\smallsmile}^i,\overline{\pi},{\smallsmile}^j)$,
and $\overline{{\wcomp}(\pi)}\neq\emptyset$ is equivalent to $\overline{\pi}\neq\imath$. Hence we have $\text{comp}({\st(\overline{\pi})})=\overline{{\wcomp}(\pi)}$ in this case, so the proof is completed.
\end{proof}

\section{Stuffle product and the Hopf algebra homomorphism $\varphi_2$}
\mlabel{sec:phi2}

This section is devoted to showing that the map $\varphi_2$ defined by Eq.~ \eqref{eq:varphisigmawgtossym} is a Hopf algebra  homomorphism from $\HSym$ to $\SSym$.

For the application in this section, we give an explicit reformulation of $\star_\lambda$ in terms of order preserving maps called the stuffle~\cite{BBBL} in the study of multiple zeta values
but could be traced back to Cartier's work~\cite{Car1972} on free commutative Rota-Baxter algebras. It also agrees with the mixable shuffle product in free Rota-Baxter algebras~\cite{Gub12,G-K1}. We will follow the presentation in~\cite{Gub12}.

For positive integers $m$, $n$ and nonnegative integer $r$ with $0\leq r\leq \min(m,n)$, define
\begin{align*}
    \mathcal{J}_{m,n,r}:=\left\{(\varphi,\psi)\left|
    \begin{array}{ccc}
    \varphi:[m]\rightarrow [m+n-r],\psi:[n]\rightarrow [m+n-r]\ \text{are order}\\
     \text{ preserving injective maps and}\   \im(\varphi)\cup\im(\psi)=[m+n-r]
    \end{array}
    \right.
    \right\}
\end{align*}
and denote
\begin{align*}
    \overline{\mathcal{J}}_{m,n}=\bigcup_{r=0}^{\min(m,n)}\mathcal{J}_{m,n,r}.
\end{align*}

Let $u=a_1a_2\cdots a_m$, $v=b_1b_2\cdots b_n$ be words on $A$. For $(\varphi,\psi)\in \mathcal{J}_{m,n,r}$, we define $u\shap_{(\varphi,\psi)}v$ to be the word of length $m+n-r$
whose $i$-th factor is
\begin{align*}
(u\shap_{(\varphi,\psi)}v)_i:=
\begin{cases}
    a_j, &\text{if}\  i=\varphi(j), i\not\in\im\psi, \\
    b_{j'}, &\text{if}\  i=\psi(j'), i\not\in\im\varphi, \\
    a_j\bullet b_{j'}, &\text{if}\  i=\varphi(j)=\psi(j').
\end{cases}
\end{align*}
In short,
\begin{align*}
    u\shap_{(\varphi,\psi)}v=(a_{\varphi^{-1}(1)}\bullet b_{\psi^{-1}(1)})(a_{\varphi^{-1}(2)}\bullet b_{\psi^{-1}(2)})\cdots (a_{\varphi^{-1}(m+n-r)}\bullet b_{\psi^{-1}(m+n-r)})
\end{align*}
with the convention that $a_{\emptyset}=b_{\emptyset}=1$. Define the \name{stuffle product} of $u$ and $v$ by
\begin{align}\label{eq:ushastlambdv}
    u\shap_{{\rm{st}},\lambda}v:=\sum_{r=0}^{\min(m,n)}\lambda^r\left(\sum_{(\varphi,\psi)\in\mathcal{J}_{m,n,r}}u\shap_{(\varphi,\psi)}v \right).
\end{align}

Completely analogous to the proof of~\cite[Theorem 3.1.19]{Gub12}, we derive
\begin{prop}\label{quasishustcoinc}
The multiplications $\shap_{{\rm {st}},\lambda}$ and $\star_{\lambda}$ are the same.
\end{prop}

Then from  Eq.~\eqref{eq:ushastlambdv}, we obtain the following alternative description of the quasi-shuffle product $\star_\lambda$, noting the left-right symmetry in the definition of $\shap_{{\rm {st}},\lambda}$.
\begin{coro}\label{coro:Rindstarlambsecondrecuform}
Let $\lambda\in \bfk$.
For all letters $c,d\in A$ and words $w,v$ on $A$, we have
\begin{align*}
wc\star_\lambda vd=(w\star_\lambda vd)c+(wc\star_\lambda v)d+\lambda(w\star_\lambda v)(c\bullet d).
\end{align*}
\end{coro}

\begin{lemma}\label{lem:varphi2sigtausp101}
Let $\sigma$ be a signed permutation of $[m]$. If there exist $i_1,i_2,i_3$ with $1\leq i_1< i_2< i_3\leq m$ such that $\sigma_{i_1}>0$, $\sigma_{i_2}<0$, $\sigma_{i_3}>0$, then for any signed permutation $\tau$, we have
$\varphi_2(\sigma\overline{\star}_{-1}\tau)=\varphi_2(\tau\overline{\star}_{-1}\sigma)=0$.
\end{lemma}
\begin{proof}
Let $\tau=\tau_1\tau_2\cdots\tau_n$. Then $\sigma\overline{\star}_{-1}\tau=\st(\sigma\star_{-1}\tau[m])$.
By Eq.~ \eqref{eq:ushastlambdv} and Proposition \ref{quasishustcoinc}, we have
\begin{align*}
   \sigma\star_{-1}\tau[m]=&\sum_{r=0}^{\min(m,n)}(-1)^r\left(\sum_{(\varphi,\psi)\in\mathcal{J}_{m,n,r}}\sigma\shap_{(\varphi,\psi)}\tau[m] \right)
\end{align*}
where for $(\varphi,\psi)\in\mathcal{J}_{m,n,r}$,
\begin{align*}
    \sigma\shap_{(\varphi,\psi)}\tau[m]
    =(\sigma_{\varphi^{-1}(1)}\bullet\tau_{\psi^{-1}(1)}[m])(\sigma_{\varphi^{-1}(2)}\bullet\tau_{\psi^{-1}(2)}[m])\cdots(\sigma_{\varphi^{-1}(m+n-r)}\bullet\tau_{\psi^{-1}(m+n-r)}[m]).
\end{align*}
Let $j_1=\varphi(i_1)$, $j_2=\varphi(i_2)$, $j_3=\varphi(i_3)$.
Since $\varphi$ is an order preserving map, we have  $1\leq j_1<j_2<j_3\leq m+n-r$  for each $(\varphi,\psi)\in\mathcal{J}_{m,n,r}$,
and hence
\begin{align*}
    \sigma\shap_{(\varphi,\psi)}\tau[m]
    =\cdots(\sigma_{i_1}\bullet\tau_{\psi^{-1}(j_1)}[m])\cdots(\sigma_{i_2}\bullet\tau_{\psi^{-1}(j_2)}[m])\cdots(\sigma_{i_3}\bullet\tau_{\psi^{-1}(j_3)}[m]) \cdots.
\end{align*}
By Eq.~\eqref{eq:diamondsp1},
\begin{align*}
   \sigma_{i_s}\bullet\tau_{\psi^{-1}(j_s)}[m]=
   \begin{cases}
   \sigma_{i_s},&\text{if}\ j_s\not\in\im \psi,\\
   0,           &\text{if}\ j_s\in\im \psi,
   \end{cases}\quad s=1,3
\end{align*}
and
\begin{align*}
   \sigma_{i_2}\bullet\tau_{\psi^{-1}(j_2)}[m]=
   \begin{cases}
   \sigma_{i_2},&\text{if either}\ j_2\not\in\im \psi\ \text{or}\ j_2\in\im \psi\ \text{with}\ \tau_{\psi^{-1}(j_2)}<0,\\
   0, &\text{if}\  j_2\in\im \psi\ {\text{with}}\ \tau_{\psi^{-1}(j_2)}>0.
   \end{cases}
\end{align*}
If either of the three products is zero, then the corresponding $\sigma\shap_{(\varphi,\psi)}\tau[m]$ is zero by definition. If none of the three is zero, then the corresponding $\sigma\shap_{(\varphi,\psi)}\tau[m]$ has two blocks of positive parts separated by a block of negative parts and hence is annihilated by $\varphi_2$, by the remark after Eq.~(\mref{eq:varphisigmawgtossym}).
Therefore, $\varphi_2(\sigma\overline{\star}_{-1}\tau)=0$.
An analogous argument shows $\varphi_2(\tau\overline{\star}_{-1}\sigma)=0$.
\end{proof}

Let $\tau$ be a signed permutation of $[n]$ with $\tau_i<0$ for all $i\in[n]$. Recall that for any  positive integer $m$, we have
$\tau[m]=(\pi_1+m^{-})(\pi_2+m^{-})\cdots(\pi_n+m^{-})$ which is denoted by ${\smallsmile}^n[m]$ for short.
Since $\star_{-1}$ is the quasi-shuffle product which coincides with the mixable shuffle product on commutative algebras, by~\cite[Prop.~3.2.2]{Gub12} or \cite[Prop~6.1]{G-K1} we have

\begin{lemma}\label{lem:0msta-10np}
For any $m,n\in\mathbb{N}$, we have
\begin{align*}
    {\smallsmile}^m\star_{-1}{\smallsmile}^n[m]=\sum_{i=0}^{\min(m,n)}(-1)^i\binom{m+n-i}{i,m-i,n-i}{\smallsmile}^{m+n-i},
\end{align*}
where $\binom{m+n-i}{i,m-i,n-i}{\smallsmile}^{m+n-i}$ means that there are $\binom{m+n-i}{i,m-i,n-i}$ signed permutations of the form ${\smallsmile}^{m+n-i}$.
\end{lemma}
As a consequence, we obtain the following identity.
\begin{coro}\label{coro:starnumbbino}
For any $m,n\in\mathbb{N}$,
\begin{align*}
\sum_{i=0}^{\min(m,n)}(-1)^i\binom{m+n-i}{i,m-i,n-i}=1.
\end{align*}
\end{coro}
\begin{proof}
Let $\sigma=\sigma_1\sigma_2\cdots\sigma_m$ and $\tau=\tau_1\tau_2\cdots\tau_n$ be signed permutations with $\sigma_i<0$ for $i\in[m]$ and $\tau_j<0$ for $j\in [n]$.
Denote by $A_{mn}$ the left-hand side of the desired equation.
By Lemma \ref{lem:0msta-10np}, $A_{mn}$ is equal to the number of terms with plus sign minus the number of terms with minus sign in the product $\sigma\star_{-1}\tau[m]$.
We show $A_{mn}=1$ by induction on $m+n$. If $m+n\leq1$, then one of $\sigma$, $\tau$ is $\imath$, so that  $A_{mn}=1$ is true. Assume that $A_{mn}=1$ holds for all $m+n\leq k$ where $k\geq1$ and consider the case $m+n=k+1$. Note that
$$
\sigma\star_{-1}\tau[m]=\sigma_1(\sigma_2\cdots\sigma_m\star_{-1}\tau[m])+(\tau_1+m^{-})(\sigma\star_{-1}\tau_2\cdots\tau_n[m])-\sigma_1(\sigma_2\cdots\sigma_m\star_{-1}\tau_2\cdots\tau_n[m]),
$$
by the induction hypothesis we have
$A_{mn}=A_{m-1,n}+A_{m,n-1}-A_{m-1,n-1}=1$.
So the proof follows by induction.
\end{proof}

\begin{lemma}\label{lem:varphi2sigtau0isigtau0j}
Let $\sigma=({\smallsmile}^{i},\overline{\sigma},{\smallsmile}^{j})$ and $\tau=({\smallsmile}^{p},\overline{\tau},{\smallsmile}^{q})$ be signed permutations where  $i,j,p$ and $q$ are nonnegative integers.
If one of $j$ or $q$ is strictly larger than $1$, then
$\varphi_2(\sigma\overline{\star}_{-1}\tau)=0$.
\end{lemma}
\begin{proof}
Let $m=\ell(\sigma)$ and $n=\ell(\tau)$. Without loss of generality, we assume that $j\geq2$. Hence $\varphi_2(\sigma)=0$
by Eq.~\eqref{eq:varphisigmawgtossym}. If $n=0$, then $\tau=\imath$ so that
$\varphi_2(\sigma\overline{\star}_{-1}\tau)=\varphi_2(\sigma)=0$.
We next suppose that $n\geq1$. Thus, $\tau\neq\imath$ and $\min(m,n)\geq1$.
Analogous to the proof of Lemma \ref{lem:varphi2sigtausp101}, we only need to show that $\varphi_2(\st(\sigma\shap_{(\varphi,\psi)}\tau[m]))=0$ for all
$(\varphi,\psi)\in\mathcal{J}_{m,n,r}$ where $0\leq r\leq \min(m,n)$.

Since $\varphi$ and $\psi$ are order  preserving injective maps with $\im(\varphi)\cup\im(\psi)=[m+n-r]$, for any $(\varphi,\psi)\in\mathcal{J}_{m,n,r}$,
we have
\begin{align}
    \varphi^{-1}(m+n-r)\in\{\emptyset,m\}, &\qquad\varphi^{-1}(m+n-r-1)\in\{\emptyset,m-1,m\},\label{eq:varpsim+n-ror-1v}
\end{align}
and
\begin{align}
    \psi^{-1}(m+n-r)\in\{\emptyset,n\}, &\qquad\psi^{-1}(m+n-r-1)\in\{\emptyset,n-1,n\}.\label{eq:varpsim+n-ror-1p}
\end{align}
However,  for $k\in[m+n-r]$, $\varphi^{-1}(k)=\emptyset$ and $\psi^{-1}(k)=\emptyset$ can not hold at the same time.

If $\overline{\tau}=\imath$ or $q\geq2$, then it follows from Eqs.~\eqref{eq:varpsim+n-ror-1v} and \eqref{eq:varpsim+n-ror-1p} that
 both
$$(\sigma\shap_{(\varphi,\psi)}\tau[m])_{m+n-r-1}=\sigma_{\varphi^{-1}(m+n-r-1)}\bullet\tau_{\psi^{-1}(m+n-r-1)}[m]$$ and $$(\sigma\shap_{(\varphi,\psi)}\tau[m])_{m+n-r}=\sigma_{\varphi^{-1}(m+n-r)}\bullet\tau_{\psi^{-1}(m+n-r)}[m]$$
are in $-{\mathbb{P}}$. Thus,
 Eq.~\eqref{eq:varphisigmawgtossym} yields  $\varphi_2(\sigma\shap_{(\varphi,\psi)}\tau[m])=0$ for all $(\varphi,\psi)\in\mathcal{J}_{m,n,r}$. Hence $\varphi_2(\sigma\overline{\star}_{-1}\tau)=0$.

If $\overline{\tau}\neq\imath$ and $0\leq q\leq 1$, then we assume that $\tau=({\smallsmile}^{p},\tau_{p+1},\cdots,\tau_{p+t},{\smallsmile}^{q})$ where $t\geq1$ and $\tau_{p+1},\cdots,\tau_{p+t}$ are positive integers.
Thus, $\overline{\tau}=(\tau_{p+1},\cdots,\tau_{p+t})$.
We will show that $\varphi_2(\sigma\overline{\star}_{-1}\tau)=0$ by considering the
following three cases.

{\bf Case 1.} $\overline{\sigma}\neq\imath$. Then, by $j\geq 2$, we may assume that
$\sigma=\sigma_{1}\cdots\sigma_{r}\cdots\sigma_{m-1}\sigma_{m}$ where $\sigma_{r}>0$, $\sigma_{m-1}<0$ and $\sigma_m<0$.
For any $(\varphi,\psi)\in\mathcal{J}_{m,n,r}$, if $\st(\sigma\shap_{(\varphi,\psi)}\tau[m])\neq0$, then there are two subcases.

If $\varphi(m-1)\leq \psi(p+t)$, then, by Eq. \eqref{eq:diamondsp1}, we must have $\varphi(r)<\varphi(m-1)< \psi(p+t)$ and
\begin{align*}
(\sigma\shap_{(\varphi,\psi)}\tau[m])_{\varphi(r)}=\sigma_r>0,\
(\sigma\shap_{(\varphi,\psi)}\tau[m])_{\varphi(m-1)}=\sigma_{m-1}<0,\  (\sigma\shap_{(\varphi,\psi)}\tau[m])_{\psi(p+t)}=\tau_{p+t}>0.
\end{align*}
So, by Eq. \eqref{eq:varphisigmawgtossym}, $\varphi_2(\sigma\shap_{(\varphi,\psi)}\tau[m])=0$.

If $\psi(p+t)<\varphi(m-1)$, then $\psi(p+t)<\varphi(m-1)<\varphi(m)$, so Eqs.~\eqref{eq:varpsim+n-ror-1v} and \eqref{eq:varpsim+n-ror-1p}
guarantee that
\begin{align*}
(\sigma\shap_{(\varphi,\psi)}\tau[m])_{m+n-r-1}<0\ \text{and}\ (\sigma\shap_{(\varphi,\psi)}\tau[m])_{m+n-r}<0
\end{align*}
Again by  Eq. \eqref{eq:varphisigmawgtossym},  $\varphi_2(\sigma\shap_{(\varphi,\psi)}\tau[m])=0$.
Therefore, we always have $\varphi_2(\sigma\overline{\star}_{-1}\tau)=0$.

{\bf Case 2.} $\overline{\sigma}=\imath$ and $q=1$. Then $\sigma={\smallsmile}^m$ with $m\geq2$ and $\tau=({\smallsmile}^{p},\overline{\tau},{\smallsmile})$.
By  Corollary \ref{coro:Rindstarlambsecondrecuform},
we have
\begin{align*}
    \sigma\overline{\star}_{-1}\tau=&\st({\smallsmile}^m\star_{-1}({\smallsmile}^{p},\overline{\tau},{\smallsmile})[m])\\
    =&\st(({\smallsmile}^{m-1}\star_{-1}({\smallsmile}^{p},\overline{\tau},{\smallsmile})[m]){\smallsmile})+\st(({\smallsmile}^{m}\star_{-1}({\smallsmile}^{p},\overline{\tau})[m]){\smallsmile})
    -\st(({\smallsmile}^{m-1}\star_{-1}({\smallsmile}^{p},\overline{\tau})[m]){\smallsmile}).
\end{align*}
It follows from $m\geq2$ that  $\st(({\smallsmile}^{m-1}\star_{-1}({\smallsmile}^{p},\overline{\tau},{\smallsmile})[m]){\smallsmile})$ must be a linear combination of signed permutations of the form
$(\cdots,\smallsmile^k)$ with $k\geq2$, so, by Eq.~\eqref{eq:varphisigmawgtossym},
$\varphi_2(\st(({\smallsmile}^{m-1}\star_{-1}({\smallsmile}^{p},\overline{\tau},{\smallsmile})[m]){\smallsmile}))=0$. Hence
\begin{align*}
    \varphi_{2}(\sigma\overline{\star}_{-1}\tau)=&\varphi_{2}(\st(({\smallsmile}^{m}\star_{-1}({\smallsmile}^{p},\overline{\tau})[m]){\smallsmile}))
    -\varphi_{2}(\st(({\smallsmile}^{m-1}\star_{-1}({\smallsmile}^{p},\overline{\tau})[m]){\smallsmile})).
\end{align*}
We iterate the above argument up to $t$ times to obtain
\begin{align*}
    \varphi_2(\sigma\overline{\star}_{-1}\tau)
    =&\varphi_2(\st(({\smallsmile}^{m}\star_{-1}{\smallsmile}^{p}[m])\overline{\tau}[m]{\smallsmile}))
    -\varphi_2(\st(({\smallsmile}^{m-1}\star_{-1}{\smallsmile}^{p})\overline{\tau}[m]{\smallsmile})).
\end{align*}
Then from Lemma \ref{lem:0msta-10np} and Corollary \ref{coro:starnumbbino}, we have
\begin{align*}
    \varphi_2(\sigma\overline{\star}_{-1}\tau)
    =&-\left(\sum_{i=0}^{\min(m,p)}(-1)^i\binom{m+p-i}{i,m-i,p-i}\right){\st(\overline\tau)}
    +\left(\sum_{i=0}^{\min(m-1,p)}(-1)^i\binom{m+p-1-i}{i,m-1-i,p-i}\right){\st(\overline\tau)}
    =0.
\end{align*}

{\bf Case 3.} $\overline{\sigma}=\imath$ and $q=0$. Then $\sigma={\smallsmile}^m$ with $m\geq2$ and $\tau=({\smallsmile}^{p},\overline{\tau})$.  By  Corollary \ref{coro:Rindstarlambsecondrecuform},
we have
\begin{align*}
    \sigma\overline{\star}_{-1}\tau=&\st({\smallsmile}^m\star_{-1}({\smallsmile}^{p},\overline{\tau})[m])\\
    =&\st(({\smallsmile}^{m-1}\star_{-1}({\smallsmile}^{p},\overline{\tau})[m]){\smallsmile})
    +\st(({\smallsmile}^{m}\star_{-1}({\smallsmile}^{p},\tau_{p+1}\cdots\tau_{p+t-1}[m])\tau_{p+t}[m]).
\end{align*}
Completely analogous to the proof in Case $1$, we obtain
\begin{align*}
    \varphi_2(\sigma\overline{\star}_{-1}\tau)
    =&\varphi_2(\st(({\smallsmile}^{m-1}\star_{-1}{\smallsmile}^{p}[m])\overline{\tau}[m]){\smallsmile})
    +\varphi_2(\st(({\smallsmile}^{m}\star_{-1}{\smallsmile}^{p})\overline{\tau}[m]))=-{\st(\overline\tau)}+{\st(\overline\tau)}=0,
\end{align*}
as required.
\end{proof}

\begin{lemma}\label{lem:sigmast-1ol1}
For any signed permutation $\sigma$, we have
$\varphi_2(\sigma\overline{\star}_{-1}1^{-})=\varphi_2(1^{-}\overline{\star}_{-1}\sigma)=0$.
\end{lemma}
\begin{proof}
Let $\sigma$ be a signed permutation of $[m]$.  If $\sigma_{i}<0$ for all $i\in [m]$, then by Lemma \ref{lem:0msta-10np} and the definition of $\varphi_2$, we have
$\varphi_2(\sigma\overline{\star}_{-1}1^{-})=\varphi_2(1^{-}\overline{\star}_{-1}\sigma)=0$; otherwise,
by Lemmas \ref{lem:varphi2sigtausp101} and \ref{lem:varphi2sigtau0isigtau0j},
it suffices to consider the cases $\sigma=({\smallsmile}^{i},\overline{\sigma})$ and $\sigma=({\smallsmile}^{i},\overline{\sigma},{\smallsmile})$, where $\overline{\sigma}\neq\imath$.
The desired identities are then proved by an elementary argument analogous to the proof of Case $2$ in  Lemma \ref{lem:varphi2sigtau0isigtau0j}.
\end{proof}

\begin{lemma}\label{lem:phi2isigdjtau=olstar-1}
Let $i,j$ be nonnegative integers and $({\smallsmile}^i,\sigma)$, $({\smallsmile}^j,\tau)$  signed permutations,
where $\sigma=\sigma_1\sigma_2\cdots\sigma_m$ and $\tau=\tau_1\tau_2\cdots\tau_n$ with $\sigma_1>0,\tau_1>0$, $m\geq1$ and $n\geq1$.
Then
\begin{align*}
    \varphi_{2}(({\smallsmile}^i,\sigma)\overline{\star}_{-1}({\smallsmile}^j,\tau))=\varphi_{2}(\st(\sigma)\overline{\star}_{-1}\st(\tau)).
\end{align*}
\end{lemma}
\begin{proof}
We use induction on $i+j$. The case $i+j=0$ is trivial. Now suppose that $i+j\geq1$ and
the claim has already been shown for all smaller values of $i+j$. Since $i+j\geq1$, without loss of generality, we may assume that
$i\geq1$. Let $k=i+m$. If $j=0$, then,
by Eq.~\eqref{eq:inductquasishuffleprod},
\begin{align*}
    ({\smallsmile}^i,\sigma)\overline{\star}_{-1}({\smallsmile}^j,\tau)=&\st(({\smallsmile}^i,\sigma)\star_{-1}\tau[k])\\
    =&\st({\smallsmile}(({\smallsmile}^{i-1},\sigma)\star_{-1}\tau[k]))+\st(\tau_{1}[k](({\smallsmile}^i,\sigma)\star_{-1}(\tau_2\cdots\tau_n)[k])).
\end{align*}
Since $\tau_1>0,\sigma_1>0$ and $i\geq1$, by Eq.~\eqref{eq:varphisigmawgtossym},
\begin{align*}
    \varphi_2(({\smallsmile}^i,\sigma)\overline{\star}_{-1}({\smallsmile}^j,\tau))
    =\varphi_2(\st({\smallsmile}(({\smallsmile}^{i-1},\sigma)\star_{-1}\tau[k])))=\varphi_2(\st({\smallsmile}^{i-1},\sigma)\overline{\star}_{-1}\tau).
\end{align*}
By the induction hypothesis, we obtain $\varphi_{2}(({\smallsmile}^i,\sigma)\overline{\star}_{-1}({\smallsmile}^j,\tau))=\varphi_{2}(\st(\sigma)\overline{\star}_{-1}\st(\tau))$.

If $j\geq1$, then Eq.~\eqref{eq:inductquasishuffleprod} yields
\begin{align*}
    ({\smallsmile}^i,\sigma)\overline{\star}_{-1}({\smallsmile}^j,\tau)
    =&\st({\smallsmile}(({\smallsmile}^{i-1},\sigma)\star_{-1}({\smallsmile}^j,\tau)[k]))\\
    &+\st({\smallsmile}(({\smallsmile}^{i},\sigma)\star_{-1}({\smallsmile}^{j-1},\tau)[k]))-\st({\smallsmile}(({\smallsmile}^{i-1},\sigma)\star_{-1}({\smallsmile}^{j-1},\tau)[k])).
\end{align*}
Then by  Eq.~\eqref{eq:varphisigmawgtossym} and the induction hypothesis, we obtain
\begin{align*}
    \varphi_2(({\smallsmile}^i,\sigma)\overline{\star}_{-1}({\smallsmile}^j,\tau))
    =&\varphi_2(\st({\smallsmile}^{i-1},\sigma)\overline{\star}_{-1}({\smallsmile}^j,\tau))
    +\varphi_2(({\smallsmile}^{i},\sigma)\overline{\star}_{-1}\st({\smallsmile}^{j-1},\tau))\\
    &-\varphi_2(\st({\smallsmile}^{i-1},\sigma)\overline{\star}_{-1}\st({\smallsmile}^{j-1},\tau))\\
    =&\varphi_2(\st(\sigma)\overline{\star}_{-1}\st(\tau))+\varphi_2(\st(\sigma)\overline{\star}_{-1}\st(\tau))-\varphi_2(\st(\sigma)\overline{\star}_{-1}\st(\tau))\\
    =&\varphi_2(\st(\sigma)\overline{\star}_{-1}\st(\tau)),
\end{align*}
as required.
\end{proof}

\begin{prop}\label{prop:varphi2isalghom}
The $\bfk$-linear map $\varphi_2:\HSym\rightarrow \SSym$
is an algebra homomorphism.
\end{prop}
\begin{proof}
It suffices to show that $\varphi_2(\sigma\overline{\star}_{-1}\tau)=\varphi_2(\sigma)\overline{\shap}\varphi_2(\tau)$ for all signed permutations $\sigma\in \mathfrak{B}_m$ and $\tau\in \mathfrak{B}_n$.
If one of $\sigma$ and $\tau$, say $\sigma$, is not of the form $({\smallsmile}^{i},\overline{\sigma},{\smallsmile}^{j})$ for some $i\in\mathbb{N}$, $j\in\{0,1\}$ and $\overline{\sigma}\neq\imath$, then from Eq.~\eqref{eq:varphisigmawgtossym} it follows that $\varphi_2(\sigma)=0$. On the other hand, by Lemmas \ref{lem:varphi2sigtausp101}, \ref{lem:varphi2sigtau0isigtau0j} and \ref{lem:sigmast-1ol1}, we have
$\varphi_2(\sigma\overline{\star}_{-1}\tau)=0$ so that $\varphi_2(\sigma\overline{\star}_{-1}\tau)=\varphi_2(\sigma)\overline{\shap}\varphi_2(\tau)$.

Now we assume that $\sigma=({\smallsmile}^{i},\overline{\sigma},{\smallsmile}^{j})$ and $\tau=({\smallsmile}^{p},\overline{\tau},{\smallsmile}^{q})$ where $j\in\{0,1\}$, $q\in\{0,1\}$,  $\overline{\sigma}\neq\imath$ and $\overline{\tau}\neq\imath$. By Lemma \ref{lem:phi2isigdjtau=olstar-1}, we can further assume that $i=p=0$,
and write $\sigma=(\sigma_1\sigma_2\cdots\sigma_s,{\smallsmile}^{j})$, $\tau=(\tau_1\tau_2\cdots\tau_t,{\smallsmile}^{q})$ for some positive integers $s$ and $t$,
where $\overline{\sigma}=\sigma_1\sigma_2\cdots\sigma_s$ and $\overline{\tau}=\tau_1\tau_2\cdots\tau_t$.

If $j=q=0$, then $\sigma=\overline{\sigma}$ and $\tau=\overline{\tau}$ are permutations so that
$\sigma\overline{\star}_{-1}\tau=\sigma\overline{\shap}\tau$ are linear combinations of permutations, and hence $\varphi_2(\sigma\overline{\star}_{-1}\tau)=\varphi_2(\sigma\overline{\shap}\tau)=\sigma\overline{\shap}\tau=\varphi_2(\sigma)\overline{\shap}\varphi_2(\tau)$.

If exactly one of $j$ and $q$ is $1$, say, $j=1$ and $q=0$, then
\begin{align*}
    \sigma\overline{\star}_{-1}\tau=&\st((\sigma_1\sigma_2\cdots\sigma_s,{\smallsmile})\star_{-1}(\tau_1\tau_2\cdots\tau_t)[m])\\
    =&\st((\sigma_1\sigma_2\cdots\sigma_s\star_{-1}\tau[m]){\smallsmile})
    +\st(((\sigma_1\sigma_2\cdots\sigma_s,{\smallsmile})\star_{-1}(\tau_1\tau_2\cdots\tau_{t-1})[m])\tau_{t}[m]).
\end{align*}
Note that $\sigma_1,\sigma_2,\cdots,\sigma_s$ are positive integers and $\tau=\overline{\tau}$ is a permutation. So it follows from Eq.~\eqref{eq:varphisigmawgtossym} that
\begin{align*}
    \varphi_2(\sigma\overline{\star}_{-1}\tau)
    =&\varphi_2(\st((\sigma_1\sigma_2\cdots\sigma_s\star_{-1}\tau[m]){\smallsmile}))\\
    =&-\varphi_2(\st(\sigma_1\sigma_2\cdots\sigma_s\shap\tau[m]))\\
    =&-\st(\sigma_1\sigma_2\cdots\sigma_s)\overline{\shap}\tau\\
    =&\varphi_2(\sigma)\overline{\shap}\varphi_2(\tau).
\end{align*}

If $j=q=1$, then by Corollary \ref{coro:Rindstarlambsecondrecuform},
\begin{align*}
    \sigma\star_{-1}\tau[m]=&(\overline{\sigma},{\smallsmile})\star_{-1}(\overline{\tau},{\smallsmile})[m]\\
    =&(\overline{\sigma}\star_{-1}(\overline{\tau},{\smallsmile})[m]){\smallsmile}+((\overline{\sigma},{\smallsmile})\star_{-1}\overline{\tau}[m]){\smallsmile}
    -(\overline{\sigma}\star_{-1}\overline{\tau}[m]){\smallsmile}.
\end{align*}
Since $s$ and $t$ are positive integers, we have
\begin{align*}
    (\overline{\sigma}\star_{-1}(\overline{\tau},{\smallsmile})[m]){\smallsmile}=(\sigma_1\cdots\sigma_{s-1}\star_{-1}(\overline{\tau},{\smallsmile})[m])\sigma_s{\smallsmile}
    +(\overline{\sigma}\star_{-1}\overline{\tau}[m]){\smallsmile}^2.
\end{align*}
Applying Eq.~\eqref{eq:varphisigmawgtossym} we conclude that $\varphi_2(\st((\overline{\sigma}\star_{-1}(\overline{\tau},{\smallsmile})[m]){\smallsmile}))=0$.
A completely analogous argument yields that
$\varphi_2(\st(((\overline{\sigma},{\smallsmile})\star_{-1}\overline{\tau}[m]){\smallsmile}))=0$. Thus,
\begin{align*}
    \varphi_2(\sigma\overline{\star}_{-1}\tau)=-\varphi_2(\st((\overline{\sigma}\star_{-1}\overline{\tau}[m]){\smallsmile}))
    =\varphi_2(\st(\overline{\sigma}\star_{-1}\overline{\tau}[m])).
\end{align*}
Note that $\overline{\sigma}$ and $\overline{\tau}$ are words on $\PP$, so
$$
\st(\overline{\sigma}\star_{-1}\overline{\tau}[m])=\st(\overline{\sigma}\shap\overline{\tau}[m])
=\st(\overline{\sigma})\overline{\shap}\st(\overline{\tau}),
$$
and hence Eq.~\eqref{eq:varphisigmawgtossym} guarantees that
$$
\varphi_2(\sigma\overline{\star}_{-1}\tau)=\varphi_2(\st(\overline{\sigma})\overline{\shap}\st(\overline{\tau}))
=\st(\overline{\sigma})\overline{\shap}\st(\overline{\tau})=\varphi_2(\sigma)\overline{\shap}\varphi_2(\tau),
$$
completing the proof.
\end{proof}

\begin{theorem}
The $\bfk$-linear map $\varphi_2:\HSym\rightarrow \SSym$
is a Hopf algebra homomorphism.
\label{thm:phi2}
\end{theorem}
\begin{proof}
By~\cite[Lemma 4.0.4]{Swe}, a bialgebra homomorphism between two Hopf algebras is automatically a Hopf algebra homomorphism.
According to Theorem \ref{thm:WSisahopfalg}, $\HSym$ is a Hopf algebra; by Proposition \ref{prop:varphi2isalghom}, $\varphi_2$ is an algebra homomorphism. Thus, it now suffices to show that $\varphi_2$ is a coalgebra homomorphism. For this we only need to show that $(\varphi_2\otimes\varphi_2)\Delta(\sigma)=\Delta(\varphi_2(\sigma))$
for all nonempty signed permutations $\sigma$.
If $\overline{\sigma}=\imath$, then $\sigma$ must be of the form $\sigma=({\smallsmile}^{\ell(\sigma)})$ with $\ell(\sigma)\geq1$, so that
$$
(\varphi_2\otimes\varphi_2)\Delta(\sigma)=\sum_{i=0}^{\ell(\sigma)}\varphi_2(\st({\smallsmile}^i))\otimes\varphi_2(\st({\smallsmile}^{\ell(\sigma)-i}))=0=\Delta(\varphi_2(\sigma)).
$$
If $\overline{\sigma}\neq\imath$, then by collecting the positive parts, we may assume that
\begin{align*}
    \sigma=({\smallsmile}^{r_1},\sigma_{r_1+1}\cdots\sigma_{r_1+s_1},\cdots,{\smallsmile}^{r_k},\sigma_{r_1+\cdots+r_k+s_1+\cdots+s_{k-1}+1}\cdots
    \sigma_{r_1+\cdots+r_k+s_1+\cdots+s_{k}},{\smallsmile}^{r_{k+1}})
\end{align*}
where $k\geq1$, $r_1\geq0$, $r_{k+1}\geq0$, $r_i\geq1$ for $2\leq i\leq k$, $s_i\geq1$ for $1\leq i\leq k$, and $\sigma_{j}\in \mathbb{P}$ for $r_1+\cdots+r_p+s_1+\cdots+s_{p-1}+1\leq j\leq r_1+\cdots+r_p+s_1+\cdots+s_p$, $1\leq p\leq k$.

If $k\geq3$, by
Eq.~\eqref{eq:varphisigmawgtossym}, it is straightforward to verify that $(\varphi_2\otimes\varphi_2)\Delta(\sigma)=0=\Delta(\varphi_2(\sigma))$.

If $k=2$, then
\begin{align*}
    \sigma=({\smallsmile}^{r_1},\sigma_{r_1+1}\cdots\sigma_{r_1+s_1},{\smallsmile}^{r_2},\sigma_{r_1+r_2+s_1+1}\cdots
    \sigma_{r_1+r_2+s_1+s_{2}},{\smallsmile}^{r_{3}}).
\end{align*}
Hence $\varphi_2(\sigma)=0$, which yields that $\Delta(\varphi_2(\sigma))=0$.
On the other hand, it follows from
Eqs.~\eqref{eq:Delta(sigma)WG} and \eqref{eq:varphisigmawgtossym} that
\begin{align*}
(\varphi_2\otimes\varphi_2)\Delta(\sigma)
=&\sum_{i=0}^{\ell(\sigma)}\varphi_2(\st(\sigma_1\cdots\sigma_i))\otimes \varphi_2(\st(\sigma_{i+1}\cdots\sigma_{\ell(\sigma)})) \\
=& \sum_{i=r_1+s_1}^{r_1+r_2+s_1}\varphi_2(\st(\sigma_1\cdots\sigma_i))\otimes \varphi_2(\st(\sigma_{i+1}\cdots\sigma_{\ell(\sigma)})) \\
=& \sum_{i=0}^{r_2}\varphi_2(\st({\smallsmile}^{r_1},\sigma_{r_1+1}\cdots\sigma_{r_1+s_1},{\smallsmile}^{i}))
\otimes \varphi_2(\st({\smallsmile}^{r_2-i},\sigma_{r_1+r_2+s_1+1}\cdots\sigma_{r_1+r_2+s_1+s_{2}},{\smallsmile}^{r_{3}})).
\end{align*}
Since $r_2\geq1$, together with Eq.~\eqref{eq:varphisigmawgtossym}, we obtain
\begin{align*}
(\varphi_2\otimes\varphi_2)\Delta(\sigma)
=\left(\sum_{i=0}^{1}(-1)^{i}\right) {\st(\sigma_{r_1+1},\cdots,\sigma_{r_1+s_1})}
\otimes\varphi_2(\st(\sigma_{r_1+r_2+s_1+1}\cdots\sigma_{r_1+r_2+s_1+s_{2}},{\smallsmile}^{r_{3}}))
=0.
\end{align*}
Hence  $(\varphi_2\otimes\varphi_2)\Delta(\sigma)=\Delta(\varphi_2(\sigma))$.

If $k=1$, then $\sigma=({\smallsmile}^{r_1},\sigma_{r_1+1}\cdots\sigma_{r_1+s_1},{\smallsmile}^{r_{2}})$.
Denote $c=\chi_{\{0,1\}}(r_2)$, that is,
$c=\begin{cases}
1,&\text{if}\ r_2\in \{0,1\},\\
0,&{\rm otherwise}.
\end{cases}
$
Then by Eq.~\eqref{eq:varphisigmawgtossym}, $\Delta(\varphi_2(\sigma))=(-1)^{r_2}c\Delta(\st(\overline{\sigma}))$, and hence
\begin{align*}
(\varphi_2\otimes&\varphi_2)\Delta(\sigma)=(\varphi_2\otimes\varphi_2)(\st\otimes\st)\left(\sum_{i=0}^{r_1}{\smallsmile}^i\otimes
({\smallsmile}^{r_1-i},\sigma_{r_1+1}\cdots\sigma_{r_1+s_1},{\smallsmile}^{r_2})\right.\\
&+\left.\sum_{i=1}^{s_1-1}({\smallsmile}^{r_1},\sigma_{r_1+1}\cdots\sigma_{r_1+i})\otimes (\sigma_{r_1+i+1}\cdots\sigma_{r_1+s_1},{\smallsmile}^{r_2})
+\sum_{i=0}^{r_2}({\smallsmile}^{r_1},\sigma_{r_1+1}\cdots\sigma_{r_1+s_1},{\smallsmile}^{i})\otimes {\smallsmile}^{r_2-i}\right)\\
=&(-1)^{r_2}c\left(1\otimes \st(\overline{\sigma})+\sum_{i=1}^{s_1-1}\st(\sigma_{r_1+1}\cdots\sigma_{r_1+i})\otimes\st(\sigma_{r_1+i+1}\cdots\sigma_{r_1+s_1})+ \st(\overline{\sigma})\otimes1\right)\\
=&(-1)^{r_2}c\Delta(\st(\overline{\sigma}))=\Delta(\varphi_2(\sigma)).
\end{align*}
Therefore, $\varphi_2$ is a coalgebra homomorphism, completing the proof.
\end{proof}

\section{Signed $P$-partitions and the Hopf algebra homomorphism $\mathcal{D}_2$}
\mlabel{sec:d2}

The purpose of this section is to prove that the linear map $\cald_2: \HSym \longrightarrow \RQSym$ defined in Eq.~\eqref{eq:defnD2} is a Hopf algebra homomorphism, by applying the notion of signed $P$-partitions which should be interesting on its own right.

\subsection{Signed $P$-partitions}\label{subsec:shadowP-partition}

We generalize Stanley's $P$-partitions~\cite{Sta12} to signed $P$-partitions, whose generating functions will be the fundamental weak quasi-symmetric functions.

A \name{signed labeled poset} $P$ of cardinality $n$ is a poset with $n$ elements labeled by distinct nonzero integers $a_1,a_2,\cdots,a_n$,
where $|a_i|\neq |a_j|$ for $i\neq j$ in $[n]$.
We let $<_P$ denote the partial order on $P$, and let $<$ denote the usual total order on the set $ \mathbb{Z}$ of integers.
Given a signed labeled poset $P$ with $n$ elements labeled by $a_1,a_2,\cdots,a_n$, taking $\st(a_1a_2\cdots a_n)=b_1b_2\cdots b_n$ gives a signed labeled poset $Q$ with its $n$ elements labeled by $b_1,b_2,\cdots,b_n$. Clearly, $b_1b_2\cdots b_n$ is a signed permutation of $[n]$ and $P$ is isomorphic to $Q$.
We call $Q$ the \name{standard signed labeled poset} of $P$ and write $Q=\st(P)$.

\begin{defn}\label{defsignedppart}
Let $P$ be a signed labeled poset.
A \name{signed $P$-partition} is a function $f:P\rightarrow \mathbb{P}$ such that for all  $i<_p j$ in $P$, we have
\begin{enumerate}
\item\label{item:defnpparta} $f(i)\leq f(j)$, and
\item\label{item:defnppartb}  $f(i)<f(j)$ whenever $i>\max(0,j)$.
\end{enumerate}
\end{defn}
Indeed, as in the case of $P$-partition \cite{Ste97},  it suffices to impose conditions \eqref{item:defnpparta} and \eqref{item:defnppartb} when $j$ covers $i$ in $P$, namely $i< j$ and no element $k\in P$ satisfies $i<k<j$.
We denote the set of all signed $P$-partitions by $A(P)$.
Notice that when the underlying set of $P$ is $[n]$,  we recover Stanley's $P$-partition up to a reversing of order.

For a word $w=a_1a_2\cdots a_n$ on the alphabet set $\mathbb{Z}\backslash\{0\}$, where $|a_i|\neq |a_j|$ if $i\neq j$, we may identify $w$ with the signed labeled poset with the total order $<_w$ defined by
\begin{align*}
a_1<_{w}a_2<_{w}\cdots<_{w}a_n.
\end{align*}
Then $A(w)$ is the set of functions $f:\{a_1,a_2,\cdots, a_n\}\rightarrow \mathbb{P}$ satisfying
\begin{align*}
f(a_1)\sim_1f(a_2)\sim_2\cdots\sim_{n-1}f(a_n),
\end{align*}
where for all $i\in[n-1]$, $\sim_i$ is $<$ if $a_i>\max(0,a_{i+1})$ and is $\leq$ otherwise.

Let $P$ be a signed labeled poset of cardinality $n$. Define $L(P)$ to be the set of linear extensions \cite{Sta12}  of $\st(P)$ that are themselves signed labeled posets.
Thus, $L(P)$ is the set of
signed permutations of $[n]$  extending $\st(P)$ to a total order; that is,
$L(P)$ is the set of signed permutations of $[n]$ such that $\pi\in L(P)$ if and only if $i<_{\st(P)} j$
implies $\pi^{-1}(i)<\pi^{-1}(j)$, as in Figure $1$.
\vspace{5mm}
\begin{center}
\setlength{\unitlength}{2mm}
\begin{picture}(6,6)
\linethickness{0.5pt}
\put(-36,0){$P:$}
\put(-29,1.5){\scriptsize{$2$}}
\put(-27.9,1.4){\line(1,-1){2}}
\put(-25.5,-2){\scriptsize{$8^-$}}
\put(-29.3,1.4){\line(-1,-1){2}}
\put(-32.5,-2){\scriptsize{$5$}}

\put(-21.5,0){$\st(P):$}
\put(-11.5,1.5){\scriptsize{$1$}}
\put(-10.4,1.4){\line(1,-1){2}}
\put(-8,-2){\scriptsize{$3^-$}}
\put(-11.8,1.4){\line(-1,-1){2}}
\put(-15,-2){\scriptsize{$2$}}

\put(-4,0){$L(P):$}
\put(3,5.5){\scriptsize{$1$}}
\put(3.4,1.7){\line(0,1){3}}
\put(3,0){\scriptsize{$2$}}
\put(3.4,-0.7){\line(0,-1){3}}
\put(3,-5.5){\scriptsize{$3^-$}}

\put(7,5.5){\scriptsize{$1$}}
\put(7.4,1.7){\line(0,1){3}}
\put(7,0){\scriptsize{$3^-$}}
\put(7.4,-0.7){\line(0,-1){3}}
\put(7,-5.5){\scriptsize{$2$}}

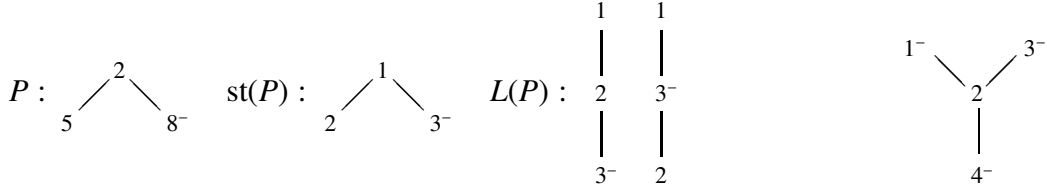
\put(-37,-9) {{Figure 1. The linear} {extensions of a  signed labeled poset.}}

\put(28,0){\scriptsize{$2$}}
\put(29,1){\line(1,1){2}}
\put(27.6,1){\line(-1,1){2}}
\put(23.5,3){\scriptsize{${1}^{-}$}}
\put(31.5,3){\scriptsize{${3}^{-}$}}
\put(28,-5.5){\scriptsize{${4}^{-}$}}
\put(28.5,-0.5){\line(0,-1){3}}
\put(15,-9){Figure 2. A  signed labeled poset.}
\end{picture}
\end{center}
\vspace{20mm}

When $P=[n]$, it follows from~\cite[Theorem 1]{Ge84} that  $A(P)=\cup_{\pi\in L(P)}A(\pi)$ where the union is disjoint.
In general we still have the union even though the union may not be disjoint, as shown in Example~\ref{exam:PS-partition}.
\begin{lemma}\label{lem:A(P)=supL(P)}
 $A(P)=\bigcup_{\pi\in L(P)}A(\pi)$.
\end{lemma}
The proof of the identity is omitted, since it is completely analogous to the proof of~\cite[Theorem 1]{Ge84} by induction on the number of incomparable pairs of elements in $P$.

\begin{exam}\label{exam:PS-partition}
Let $P$ be the signed labeled poset in Figure 2.
Then $P=\{{1}^{-},2,{3}^{-},{4}^{-}\}$ and $A(P)$ is the set of functions $f:P\rightarrow \mathbb{P}$ satisfying $f({4}^{-})\leq f(2)<f({1}^{-})$ and $f(2)<f({3}^{-})$.

It is easy to see that $L(P)=\{\pi,\sigma\}$, where $\pi={4}^{-}2{1}^{-}\,{3}^{-}$ and $\sigma={4}^{-}2{3}^{-}\,{1}^{-}$.
By the definition of signed $P$-partitions,
$$A(\pi)=\{f:P\rightarrow \mathbb{P}|f({4}^{-})\leq f(2)<f({1}^{-})\leq f({3}^{-})\}$$
and
$$A(\sigma)=\{f:P\rightarrow \mathbb{P}|f({4}^{-})\leq f(2)<f({3}^{-})\leq f({1}^{-})\}.$$
Thus,  $A(P)=A(\pi)\cup A(\sigma)$.
Note that the functions $f:P\rightarrow \mathbb{P}$ satisfying $f({4}^{-})\leq f(2)<f({1}^{-})=f({3}^{-})$
are in both $A(\pi)$ and $A(\sigma)$, so the union is not disjoint.

The signed permutations $\pi$ and $\sigma$ have the same generating function
\begin{align*}
\Gamma(\pi)=\Gamma(\sigma)=\sum_{j_1\leq j_2<j_3\leq j_4}x_{j_1}^{\varepsilon}x_{j_2}x_{j_3}^{\varepsilon}x_{j_4}^{\varepsilon}=F_{(\varepsilon,1,\varepsilon^2)},
\end{align*}
where the generating function of a poset is given by Eq.~\eqref{generatfofposet}.
Since $f\in A(\pi)\cap A(\sigma)$ is equivalent to $f({4}^{-})\leq f(2)<f({1}^{-})=f({3}^{-})$, we have
\begin{align*}
\sum_{f\in A(\pi)\cap A(\sigma)}w(f)=\sum_{j_1\leq j_2<j_3=j_4}x_{j_1}^{\varepsilon}x_{j_2}x_{j_3}^{\varepsilon}x_{j_4}^{\varepsilon}
=\sum_{j_1\leq j_2<j_3}x_{j_1}^{\varepsilon}x_{j_2}x_{j_3}^{\varepsilon}=F_{(\varepsilon,1,\varepsilon)}
\end{align*}
and therefore
$\Gamma(P)=2F_{(\varepsilon,1,\varepsilon^2)}-F_{(\varepsilon,1,\varepsilon)}$.
\end{exam}

Let $X=\{x_n\,|\,n\in \PP\}$ be a set of commuting variables.
For a signed labeled poset $P$, we define the \name{weight} of $f\in A(P)$  to be the $\widetilde{\NN}$-exponent monomials
\begin{align*}
w(f):=\prod_{i\in P}x_{f(i)}^{\delta(i,\PP)},\  \ {\rm{where}}\
\delta(i,\PP)=\begin{cases}
1 ,&\text{if}\  i>0,\\
\varepsilon,&\text{otherwise},
\end{cases}
\end{align*}
and define the \name{generating function} for $P$ to be
\begin{align}\label{generatfofposet}
\Gamma(P):=\sum_{f\in A(P)}w(f)\in \bfk[[X]]_{\widetilde{\mathbb{N}}}.
\end{align}
Thus, for a signed permutation $\pi$ of $[n]$, the generating function
\begin{align}\label{eq:Gamma(pi)weightposet}
\Gamma(\pi)=\sum_{f\in A(\pi)}w(f)=\sum_{f\in A(\pi)}\prod_{i=1}^nx_{f(\pi_{i})}^{\delta(\pi_{i},\PP)}
\end{align}
is a weak quasi-symmetric function. Then the same holds for
$\Gamma(P)$ by Lemma~\ref{lem:A(P)=supL(P)}. Clearly, $\Gamma(P)=\Gamma(\st(P))$.
It is convenient to extend the notation a little further: If $\Lambda$ is a set of signed permutations,
then we set $\Gamma(\sum_{\pi\in\Lambda}\pi)=\sum_{\pi\in\Lambda}\Gamma(\pi)$.

The power series of the form $\Gamma(\pi)$ are important. It is easy to see that $\Gamma(\pi)$ depends only on the regularized composition of $\pi$.

\begin{theorem}\label{thm:Gamma(pi)=Fwcomp(pi)}
For any signed permutation $\pi$, we have $\Gamma(\pi)=F_{{\wcomp}(\pi)}$ for  $F_{{\wcomp}(\pi)}$ in Eq.~(\mref{eqdefn:fundmbasis1}).
\end{theorem}
\begin{proof}
Let $\pi$ be a signed permutation of $[n]$ and let $m_i=f(\pi_{i})$, $i\in[n]$. Then by Eq.~\eqref{eq:Gamma(pi)weightposet},
\begin{align*}
    \Gamma(\pi)=\sum_{\substack{m_1\leq m_2\leq \cdots \leq m_{n} \\ \pi_{i}>\max(0,\pi_{i+1}) \Rightarrow m_{i}<m_{i+1}}} x_{m_1}^{\delta(\pi_1,{\PP})}x_{m_2}^{\delta(\pi_2,{\PP})}\cdots x_{m_n}^{\delta(\pi_{n},{\PP})}.
\end{align*}
By Eq.~\eqref{eq:defn:descentofpi}, for each $i\in[n-1]$, $\pi_{i}>\max(0,\pi_{i+1})$ is equivalent to $i\in WD(\pi)$. Hence in view of Eq.~\eqref{eq:lemwcomppi}, the equation $WD(\pi)=D({\wcomp}(\pi))$ holds.
It follows from Eq.~\eqref{eqdefn:fundmbasis1}  that $\Gamma(\pi)=F_{{\wcomp}(\pi)}$, as required.
\end{proof}

\subsection{$\mathcal{D}_2$ is a Hopf algebra homomorphism}
\mlabel{ss:d2}

Analogous to~\cite[Proposition 4.6]{Mal94}, we have
\begin{lemma}\label{lem:GaPuniQ}
Let $P$ and $Q$ be two disjoint signed labeled posets. Then $\Gamma(P\cup Q)=\Gamma(P)\Gamma(Q)$.
\end{lemma}
\begin{proof}
Since  $P$ and $Q$ are disjoint, the correspondence $f\mapsto (f|_P,f|_Q)$ induces a bijection between $A(P\cup Q)$ and $A(P)\times A(Q)$.
It follows at once that
\begin{align*}
  \Gamma(P\cup Q)=\sum_{f\in A(P\cup Q)}\prod_{i\in P\cup Q} x_{f(i)}^{\delta(i,\PP)}
                 =\left(\sum_{g\in A(P)}\prod_{i\in P} x_{g(i)}^{\delta(i,\PP)}\right)\left(\sum_{h\in A(Q)}\prod_{j\in Q} x_{h(j)}^{\delta(j,\PP)}\right)
                 =\Gamma(P)\Gamma(Q),
\end{align*}
as required.
\end{proof}

\begin{prop}\label{prop:Gamma(sigma2mcdottau[m])}
For any signed permutations $\sigma$ and $\tau$, we have
$\Gamma(\sigma)\Gamma(\tau)=\Gamma(\sigma\overline{\star}_{-1}\tau)$.
\end{prop}
\begin{proof}
Let $\sigma=\sigma_1\cdots\sigma_m\in \mathfrak{B}_m$, $\tau=\tau_1\cdots\tau_n\in \mathfrak{B}_n$, and write $\sigma'=\sigma_2\cdots\sigma_m$, $\tau'=\tau_2\cdots\tau_n$.
By Lemma \ref{lem:GaPuniQ},
\begin{align*}
 \Gamma(\sigma)\Gamma(\tau)=\Gamma(\sigma\cup\tau[m])=\sum_{f\in A(\sigma\cup\tau[m])}w(f).
\end{align*}
Since
\begin{align*}
    A(\sigma\cup\tau[m])=A(\sigma_1(\sigma'\cup\tau[m]))\cup A(\tau_1[m](\sigma\cup \tau'[m])),
\end{align*}
we have
\begin{align}\label{eq:Gasigtau=GasigGatauhom}
 \Gamma(\sigma)\Gamma(\tau)=\Gamma(\sigma_1(\sigma'\cup\tau[m]))+\Gamma(\tau_1[m](\sigma\cup \tau'[m]))-\sum_{f\in B}w(f),
\end{align}
where $B:=A(\sigma_1(\sigma'\cup\tau[m]))\cap A(\tau_1[m](\sigma\cup \tau'[m]))$.

If $B$ is a nonempty set, then for any $f\in B$, we must have $f(\sigma_1)=f(\tau_1[m])$.
Define
$$  g(\sigma_1):= f(\sigma_1)=f(\tau_1[m]),\ g(\sigma_i):=f(\sigma_i),\ g(\tau_j[m]):=f(\tau_j[m]), 2\leq i\leq m, 2\leq j\leq n.
$$
Then $g$ is in $A(\sigma_1(\sigma'\cup \tau'[m]))$, giving rise to a bijection between $B$ and $A(\sigma_1(\sigma'\cup \tau'[m]))$.
By Definition \ref{defsignedppart}, $f(\sigma_1)=f(\tau_1[m])$ implies
$\sigma_1<0$ and $\tau_1<0$, so that $\delta(\sigma_1,\PP)=\delta(\tau_1,\PP)=\varepsilon$.
Hence, for any $f\in B$ corresponding to $g\in A(\sigma_1(\sigma'\cup \tau'[m]))$, we have
$$
w(f)=x_{f(\sigma_1)}^\varepsilon\prod_{i=2}^mx_{f(\sigma_i)}^{\delta(\sigma_i,\PP)}\cdot x_{f(\tau_1[m])}^\varepsilon\prod_{j=2}^nx_{f(\tau_j[m])}^{\delta(\tau_j[m],\PP)}
=x_{g(\sigma_1)}^\varepsilon\prod_{i=2}^mx_{g(\sigma_i)}^{\delta(\sigma_i,\PP)}\cdot \prod_{j=2}^nx_{g(\tau_j[m])}^{\delta(\tau_j[m],\PP)}=w(g).
$$
Therefore, Eq.~\eqref{eq:Gasigtau=GasigGatauhom} is equivalent to
\begin{align}\label{eq:Gasigtau=GasigGatauhom2}
 \Gamma(\sigma)\Gamma(\tau)=\Gamma(\sigma_1(\sigma'\cup\tau[m]))+\Gamma(\tau_1[m](\sigma\cup \tau'[m]))-\chi_{-\PP}(\sigma_1)\chi_{-\PP}(\tau_1)\Gamma(\sigma_1(\sigma'\cup \tau'[m])).
\end{align}
Note that $\Gamma(w)=\Gamma(\st(w))$ for any word on $-\PP\cup\PP$ with $|w_i|\neq |w_j|$ if $i\neq j$, and that Eqs.~ \eqref{eq:defnsigmacdotlambdatau} and \eqref{eq:Gasigtau=GasigGatauhom2}
satisfy  the same recursion, so
$\Gamma(\sigma)\Gamma(\tau)=\Gamma(\sigma\overline{\star}_{-1}\tau)$.
\end{proof}

Analogous to~\cite[Corollairy 4.10]{Mal94}, we obtain the product rule of two fundamental weak quasi-symmetric functions.
\begin{coro}\label{coro:productruleofweakfdqsf}
For any signed permutations $\pi,\sigma$, we have
\begin{align*}
F_{\wcomp(\pi)}F_{\wcomp(\sigma)} =\sum_{\tau}C_{\pi,\sigma}^{\tau}F_{\wcomp(\tau)},
\end{align*}
where $\tau$ range over all signed permutations and $C_{\pi,\sigma}^{\tau}$ is the coefficient of $\tau$ in the product $\pi\overline{\star}_{-1}\sigma$.
\end{coro}
In abbreviated notation, we write $F_{\wcomp(\pi)}F_{\wcomp(\sigma)} =F_{\wcomp(\pi\overline{\star}_{-1}\sigma)}$.
\begin{proof}
This follows immediately from Theorem \ref{thm:Gamma(pi)=Fwcomp(pi)} and Proposition \ref{prop:Gamma(sigma2mcdottau[m])}.
\end{proof}
For example, taking $\lambda=-1$ in Eq.~\eqref{example12st21} gives
\begin{align*}
F_{1\varepsilon}F_{1\varepsilon}=F_{\wcomp(12^{-}\overline{\star}_{-1}21^{-})}=2F_{1\varepsilon1\varepsilon}+2F_{11\varepsilon^2}+2F_{2\varepsilon^2}-F_{11\varepsilon}-F_{2\varepsilon}.
\end{align*}

\begin{lemma}\label{lem:wcomppicdotnearcons}
Let $n$ be a positive integer and $\pi$ a signed permutation of $[n]$, and let $p\in[n]$.
\begin{enumerate}
\item\label{wcompconcwp2}If $\pi_p<0$ or $\pi_p>0$ with $p\in WD(\pi)$,
then
\begin{align*}
{\wcomp}(\pi)={\wcomp}(\st(\pi_1\cdots\pi_p))\cdot{\wcomp}(\st(\pi_{p+1}\cdots\pi_n));
\end{align*}

\item\label{wcompnearconcwp1} If $\pi_p>0$ with  $p\notin WD(\pi)$,
then
$$
{\wcomp}(\pi)={\wcomp}(\st(\pi_1\cdots\pi_p))\odot{\wcomp}(\st(\pi_{p+1}\cdots\pi_n)).
$$
\end{enumerate}
\end{lemma}
\begin{proof}
Let $\pi={\smallsmile}^{i_1}\pi_{i_1+1}\cdots \pi_{j_1}{\smallsmile}^{i_2}\pi_{j_1+i_2+1}\cdots \pi_{j_2}\cdots {\smallsmile}^{i_k}\pi_{j_{k-1}+i_k+1}\cdots \pi_{j_k}{\smallsmile}^{i_{k+1}}$ where
$i_p\in\mathbb{N}$, $p\in[k+1]$, $j_{q-1}+i_q+1\leq j_{q}$, and $\pi_{l}\in\PP$ for $j_{q-1}+i_q+1\leq l\leq j_{q}$, $q\in[k]$ with the notation $j_{0}=0$.
\smallskip

\noindent
\eqref{wcompconcwp2} If $\pi_p<0$, then by Eq.~\eqref{eq:lemwcomppi}, there exists $q\in[k]$ such that $j_{q-1}<p\leq j_{q-1}+i_q$, and hence
\begin{align*}
{\wcomp}(\pi)=&(\varepsilon^{i_1},{\rm{comp}}(\st(\pi_{i_1+1}\cdots \pi_{j_1})),\cdots,\varepsilon^{i_{q-1}},{\rm{comp}}(\st(\pi_{j_{q-2}+i_{q-1}+1}\cdots \pi_{j_{q-1}})),\varepsilon^{p-j_{q-1}})\\
   & \cdot(\varepsilon^{j_{q-1}+i_q-p},\cdots,\varepsilon^{i_k},{\rm{comp}}(\st(\pi_{j_{k-1}+i_k+1}\cdots \pi_{j_k})),\varepsilon^{i_{k+1}})\\
    =&{\wcomp}(\st(\pi_1\cdots\pi_p))\cdot{\wcomp}(\st(\pi_{p+1}\cdots\pi_n)).
\end{align*}
If $\pi_p>0$, then by Eq.~\eqref{eq:lemwcomppi}, there exists $q\in[k]$ such that $j_{q-1}+i_q<p\leq j_{q}$. It follows from Eq.~\eqref{eq:defn:descentofpi} that
$p\in WD(\pi)$ is equivalent to $p-j_{q-1}-i_q\in WD(\st(\pi_{j_{q-1}+i_q+1}\cdots \pi_{j_q}))$.
Thus if $\pi_p>0$ with $p\in WD(\pi)$, then we have
$$
{\rm{comp}}(\st(\pi_{j_{q-1}+i_q+1}\cdots \pi_{j_q}))={\rm{comp}}(\st(\pi_{j_{q-1}+i_q+1}\cdots \pi_{p}))\cdot{\rm{comp}}(\st(\pi_{p+1}\cdots \pi_{j_q}))
$$
so that ${\wcomp}(\pi)={\wcomp}(\st(\pi_1\cdots\pi_p))\cdot{\wcomp}(\st(\pi_{p+1}\cdots\pi_n))$, proving Item \eqref{wcompconcwp2}.
\smallskip

\noindent
\eqref{wcompnearconcwp1} If $\pi_p>0$ with $p\notin WD(\pi)$, then by the proof of Item~\eqref{wcompconcwp2}, $p-j_{q-1}-i_q\notin WD(\st(\pi_{j_{q-1}+i_q+1}\cdots \pi_{j_q}))$,
so $$
{\rm{comp}}(\st(\pi_{j_{q-1}+i_q+1}\cdots \pi_{j_q}))={\rm{comp}}(\st(\pi_{j_{q-1}+i_q+1}\cdots \pi_{p}))\odot{\rm{comp}}(\st(\pi_{p+1}\cdots \pi_{j_q}))
$$
and hence ${\wcomp}(\pi)={\wcomp}(\st(\pi_1\cdots\pi_p))\odot{\wcomp}(\st(\pi_{p+1}\cdots\pi_n))$, proving Item~\eqref{wcompnearconcwp1}.
\end{proof}

\begin{theorem}\label{HAHD2fHSymtRQSym}
The $\bfk$-linear map $\mathcal{D}_2:\HSym\rightarrow \RQSym$ induced by $\mathcal{D}_2(\pi)=F_{{\wcomp}(\pi)}$ is a Hopf algebra homomorphism.
\end{theorem}
\begin{proof}
By Theorem \ref{thm:Gamma(pi)=Fwcomp(pi)} and Proposition \ref{prop:Gamma(sigma2mcdottau[m])}, we have
\begin{align*}
 \mathcal{D}_2(\sigma\overline{\star}_{-1}\tau)=F_{{\wcomp}(\sigma\overline{\star}_{-1}\tau)}=\Gamma(\sigma\overline{\star}_{-1}\tau)
 =\Gamma(\sigma)\Gamma(\tau)=F_{{\wcomp}(\sigma)}F_{{\wcomp}(\tau)}
 = \mathcal{D}_2(\sigma)\mathcal{D}_2(\tau)
\end{align*}
for any signed permutations $\sigma$ and $\tau$. Thus $ \mathcal{D}_2$ is an algebra homomorphism. Next we show that $ \mathcal{D}_2$ is comultiplication-preserving.
Let $\pi$ be a signed permutation of $[n]$.  Clearly $\epsilon_R\mathcal{D}_2(\pi)=\epsilon(\pi)$ by Eqs.~\eqref{HSymepde} and \eqref{deepSant}. It follows from  Eq.~\eqref{eq:Delta(sigma)WG} that
\begin{align*}
\Delta(\pi)=\sum_{p=0}^n\st(\pi_1\cdots\pi_p)\otimes\st(\pi_{p+1}\cdots\pi_n),
\end{align*}
which together with Lemma \ref{lem:wcomppicdotnearcons} implies that
\begin{align*}
    (\mathcal{D}_2\otimes \mathcal{D}_2)\Delta(\pi)=&\sum_{p=0}^nF_{{\wcomp}(\st(\pi_1\cdots\pi_p))}\otimes F_{{\wcomp}(\st(\pi_{p+1}\cdots\pi_n))}\\
    =&\sum_{{\wcomp}(\pi)=\alpha\cdot\beta}F_\alpha\otimes F_\beta+\sum_{{\wcomp}(\pi)=\alpha\odot\beta}F_\alpha\otimes F_\beta.
\end{align*}
Thus, $$(\mathcal{D}_2\otimes \mathcal{D}_2)\Delta(\pi)=\Delta_R(F_{{\wcomp}(\pi)})=\Delta_R( \mathcal{D}_2(\pi))$$ by  Eq.~\eqref{eq:antipodeforwqsymfundam}.
Therefore, $\mathcal{D}_2$ is a bialgebra homomorphism and hence a Hopf algebra homomorphism by~\cite[Lemma 4.0.4]{Swe}.
\end{proof}

\noindent

{\bf Acknowledgements.}
This work was partially supported by the National Natural Science Foundation of China $($Grant No. 11771190, 11501467$)$ and the Natural Science Foundation of Chongqing  $($Grant No. cstc2019jcyj-msxmX0435$)$.

\end{document}